\documentclass[a4paper]{article}

\usepackage[T1]{fontenc}
\usepackage[latin9]{inputenc}
\usepackage{geometry}
\usepackage{amsmath}
\usepackage{amsthm}
\usepackage{amssymb}
\usepackage{bbm}
\usepackage{color}
\usepackage{paralist}
\usepackage[unicode=true,pdfusetitle,
 bookmarks=true,bookmarksnumbered=false,bookmarksopen=false,
 breaklinks=false,pdfborder={0 0 0},pdfborderstyle={},backref=false,colorlinks=false]
 {hyperref}
 \usepackage{tikz}
 \usetikzlibrary{decorations.pathreplacing}
 
\usepackage[labelfont=bf,labelsep=period]{caption}

\makeatletter

\usepackage[nameinlink,capitalise,noabbrev]{cleveref}

\hypersetup{%
    bookmarksnumbered, bookmarksopen=true, bookmarksopenlevel=1,%
}

\theoremstyle{plain}
\newtheorem{thm}{Theorem}[section]
\crefname{thm}{Theorem}{Theorems}
\theoremstyle{plain}
\newtheorem{lem}[thm]{Lemma}
\crefname{lem}{Lemma}{Lemmas}
\theoremstyle{plain}

\theoremstyle{plain}
\newtheorem*{claim*}{Claim}
\crefname{claim}{Claim}{Claims}
\theoremstyle{definition}
\newtheorem{defn}[thm]{Definition}
\theoremstyle{plain}

\crefname{conjecture}{Conjecture}{Conjectures}

\theoremstyle{plain}

\crefname{prop}{Proposition}{Propositions}
\theoremstyle{definition}

\theoremstyle{definition}

\theoremstyle{plain}
\newtheorem{claim}[thm]{Claim}
\newtheorem{fact}[thm]{Fact}
\numberwithin{equation}{section}
\crefformat{equation}{#2(#1)#3}
\crefname{appsec}{Appendix}{Appendices}

\date{}

\crefformat{equation}{#2(#1)#3}

\let\originalleft\left
\let\originalright\right
\renewcommand{\left}{\mathopen{}\mathclose\bgroup\originalleft}
\renewcommand{\right}{\aftergroup\egroup\originalright}
\usepackage{verbatim}

\makeatletter
\renewcommand*{\UrlTildeSpecial}{%
  \do\~{%
    \mbox{%
      \fontfamily{ptm}\selectfont
      \textasciitilde
    }%
  }%
}%
\let\Url@force@Tilde\UrlTildeSpecial
\makeatother

\let\OLDthebibliography\thebibliography
\renewcommand\thebibliography[1]{
  \OLDthebibliography{#1}
  \setlength{\parskip}{0pt}
  \setlength{\itemsep}{3pt plus 0.3ex}
}

\makeatother

\allowdisplaybreaks

\begin{document}
\title{Extension complexity of low-dimensional polytopes}
\author{
Matthew Kwan\thanks{IST Austria, Klosterneuburg, Austria.
Email: \href{matthew.kwan@ist.ac.at}{\nolinkurl{matthew.kwan@ist.ac.at}}.
Research supported by SNSF Project 178493 and NSF Award DMS-1953990.}
\and
Lisa Sauermann\thanks{Department of Mathematics, Massachusetts Institute of Technology, Cambridge, MA.
Email: \href{lsauerma@mit.edu}{\nolinkurl{lsauerma@mit.edu}}. Research supported by NSF Award DMS-1953772.}
\and
Yufei Zhao\thanks{Department of Mathematics, Massachusetts Institute of Technology, Cambridge, MA.
Email: \href{yufeiz@mit.edu}{\nolinkurl{yufeiz@mit.edu}}. 
Research supported by NSF Award DMS-1764176, NSF CAREER Award DMS-2044606, a Sloan Research Fellowship, and the MIT Solomon Buchsbaum Fund.}
}

\maketitle
\global\long\def\RR{\mathbb{R}}%
\global\long\def\QQ{\mathbb{Q}}%
\global\long\def\E{\mathbb{E}}%
\global\long\def\Var{\operatorname{Var}}%
\global\long\def\Vol{\operatorname{Vol}}%
\global\long\def\CC{\mathbb{C}}%
\global\long\def\NN{\mathbb{N}}%
\global\long\def\supp{\operatorname{supp}}%
\global\long\def\one{\boldsymbol{1}}%
\global\long\def\d{\operatorname{d}}%
\global\long\def\rank{\operatorname{rank}}%
\global\long\def\conv{\operatorname{conv}}%
\global\long\def\xc{\operatorname{xc}}%
\global\long\def\floor#1{\left\lfloor #1\right\rfloor }%
\global\long\def\ceil#1{\left\lceil #1\right\rceil }%
\global\long\def\cond{\,\middle|\,}%
\global\long\def\su{\subseteq}%
\global\long\def\eps{\varepsilon}%

\begin{abstract}
Sometimes, it is possible to represent a complicated polytope as a projection of a much simpler polytope. To quantify this phenomenon, the \emph{extension complexity} of a polytope $P$ is defined to be the minimum number of facets of a (possibly higher-dimensional) polytope from which $P$ can be obtained as a (linear) projection. 
This notion is motivated by its relevance to combinatorial optimisation, and has been studied intensively for various specific polytopes associated with important optimisation problems.
In this paper we study extension complexity as a parameter of general polytopes, more specifically considering various families of low-dimensional polytopes.

First, we prove that for a fixed dimension $d$, the extension complexity of a random $d$-dimensional polytope (obtained as the convex hull of random points in a ball or on a sphere) is typically on the order of the square root of its number of vertices. Second, we prove that any cyclic $n$-vertex polygon (whose vertices lie on a circle) has extension complexity at most $24\sqrt n$. This bound is tight up to the constant factor $24$. Finally, we show that there exists an $n^{o(1)}$-dimensional polytope with at most $n$ vertices and extension complexity $n^{1-o(1)}$. Our theorems are proved with a range of different techniques, which we hope will be of further interest.
\end{abstract}

\section{Introduction}
A regular hexagon $P$ is an example of a two-dimensional polytope.
It has six facets, which means that we need at least six linear constraints when describing $P$ by a list of inequalities. However, a curious observation is that we can actually view $P$ as a projection of a three-dimensional polytope having only \emph{five} facets (see \cref{fig:projection}). Actually, it follows from work of Ben-Tal and Nemirovski~\cite{BN01} (see also \cite{KP11}) that a regular $n$-gon can be described as a linear projection of a polytope in some higher dimension which has only $O(\log n)$ facets.

\begin{figure}[t]
\begin{center}\includegraphics[scale=0.35,trim={0 2.5cm 0 2cm},clip]{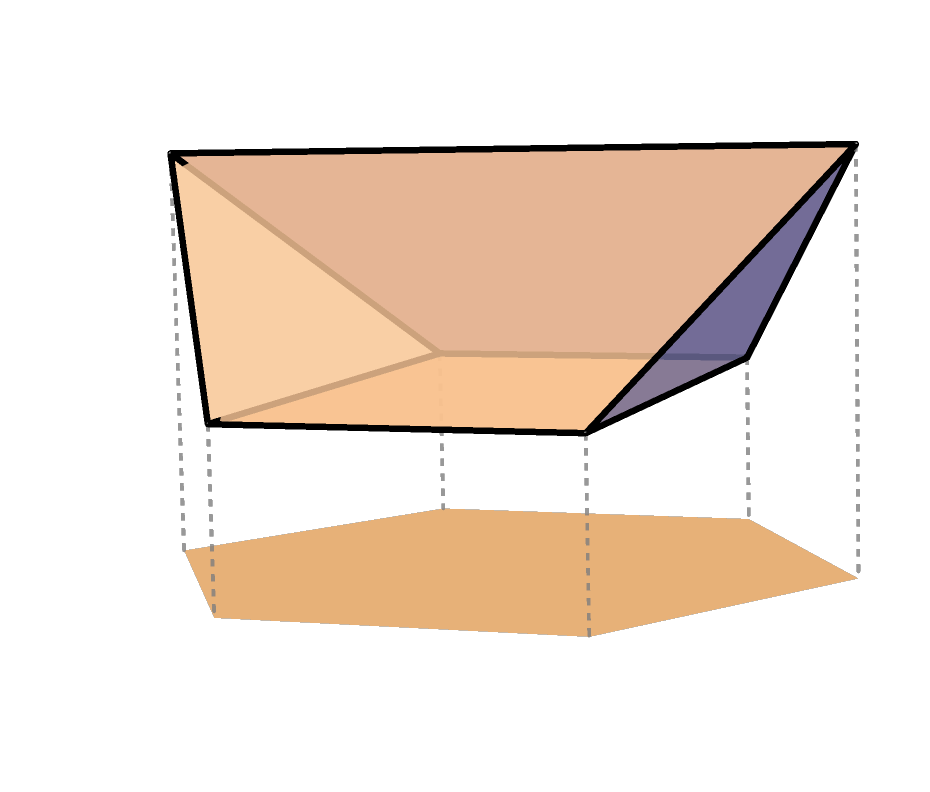}
\end{center}\caption{\label{fig:projection}A regular hexagon can be represented as the projection of a $3$-dimensional polytope with only five facets.}
\end{figure}

That is to say, sometimes it is possible to represent a polytope $P$
with a large number of facets as a projection of a higher-dimensional
polytope $P'$ with a much smaller number of facets. This observation
is enormously useful in combinatorial optimisation,
because it can allow one to solve a linear program with many constraints
via a linear program with a much smaller number of constraints (this
latter linear program is called an \emph{extended formulation}). To quantify this phenomenon, the \emph{extension complexity} $\xc(P)$ of a $d$-dimensional polytope $P$ is defined to be the minimum number of facets
of a polytope $P'\subseteq\RR^{d'}$ such that one can obtain $P$
as the image of $P'$ under a projection onto a $d$-dimensional subspace.

The study of extended formulations and extension complexity has a rich history (see for example the surveys \cite{CCZ13,Kai11,VW09}), and has enjoyed particular attention over the last decade. A large part of the research in this area has focused on understanding the extension complexity of specific polytopes associated with important optimisation problems, such as the max-cut problem~\cite{CLRS16}, the travelling salesman problem~\cite{FMPTdW15} and the perfect matching problem~\cite{Rot17}. In contrast, in this paper we are interested in more theoretical aspects of extension complexity as a parameter of general polytopes.

In his foundational paper \cite{Yan91}, Yannakakis discovered a fundamental connection between extension complexity and the notion of \emph{nonnegative rank}. For a nonnegative $m\times n$ matrix $M\in \mathbb{R}_{\geq 0}^{m\times n}$, we define the nonnegative rank of $M$, denoted $\rank_{+}M$, to be the minimum $r$ such that there is a factorisation $M=TU$, where $T\in \mathbb{R}_{\geq 0}^{m\times r}$ and $U\in \mathbb{R}_{\geq 0}^{r\times n}$ are nonnegative matrices with $r$ columns and $r$ rows, respectively. Yannakakis showed that the extension complexity of a polytope $P$ is equal to the nonnegative rank of a certain matrix (a \emph{slack matrix}) associated with $P$, and the study of extension complexity is therefore closely related to the study of nonnegative rank. It is worth remarking that the notion of nonnegative rank also plays an important role in machine learning and statistics, as well as in communication complexity (see for example the survey \cite{Gil15}).

It is a very difficult problem to compute the nonnegative rank of
a given nonnegative matrix\footnote{In fact, it is not immediately obvious that there is any algorithm
that runs in any finite amount of time! This was first proved by Cohen and
Rothblum~\cite{CR93}. The current state of the art is an algorithm
due to Moitra~\cite{Moi16} that runs in exponential time.
Some reductions to canonical computationally difficult problems were proved in \cite{AGKM16,Shi18,Vav09}.}, and it also seems to be very difficult to determine the extension complexity of a given polytope. However, it is easy to show that the extension complexity $\xc(P)$ of a polytope $P$ is at most the number of facets of $P$, and also at most the number of vertices of $P$. In
fact, due to the existence of an operation called the \emph{polar
dual}, which flips the roles of vertices and facets of a polytope and does not affect the extension complexity, vertices and facets are basically interchangeable from the point of view of extension complexity.

It is natural to ask to which extent the \emph{dimension} of a polytope controls its extension complexity. For example, if a polytope $P$ has $n$ vertices and some small dimension $d$, can we give a stronger upper bound than $n$ on its extension complexity? What if $P$ is in some sense a ``generic'' or ``random'' polytope of dimension $d$? Various questions of this type (and similar questions in the equivalent setting of nonnegative rank) have been asked over the years, in online media such as the \emph{Open Problem Garden}~\cite{OPG}, at conferences in mathematics and computer science (see for example~\cite{Dag13,Dag15,Lin13}), and in a large number of papers (see for example \cite{BL09,BP13,FRT12,Hru12,LC10,PP15,Shi19,Sh14v1,Sh14v2,VGGT16}).

In this paper we make several contributions towards answering these questions. First, for constant $d$ we consider two natural models of random $d$-dimensional polytopes, namely polytopes obtained as the convex hull of $n$ independent uniformly random points on the unit sphere or $m$ independent uniformly random points in the unit ball. For both of these models, we show that the extension complexity is likely to be about the square root of the number of vertices. The important part here is the upper bound: Padrol~\cite{Pad16} has already shown that for a wide range of different notions of random polytopes and any $d\ge 2$, a random $d$-dimensional polytope with at least $n$ vertices or facets typically has extension complexity at least $\Omega(\sqrt{n})$ (earlier, Fiorini, Rothvo{\ss} and Tiwary~\cite{FRT12} proved a very similar result, but stated it only for $d=2$). In contrast, the upper bounds in our results are new, and (at least for dimension $d\geq 3$) no nontrivial upper bounds were known in this setting before.

\begin{thm}
\label{thm:sphere}
Fix $d\geq 2$ and let $P$ be the convex hull
of $n$ random points on the $(d-1)$-dimensional unit
sphere $S\subseteq\RR^{d}$.Then, a.a.s.\footnote{By ``asymptotically almost surely'', or ``a.a.s.'', we mean that
the probability of an event is $1-o\left(1\right)$. Here and for
the rest of the paper, asymptotics are as $n\to\infty$ (for $d$ fixed). By the asymptotic notation $\xc(P)=\Theta(\sqrt{n})$ we mean that there exist positive constants $C$ and $c$ (which may depend on $d$) such that $c\sqrt{n}\leq \xc(P)\leq C\sqrt{n}$ for all (sufficiently large) $n$.}\ $\xc(P)=\Theta(\sqrt{n})$.
\end{thm}

\begin{thm}
\label{thm:ball}
Fix $d\geq 2$, let $P$ be the convex hull of
$m$ random points in the $d$-dimensional unit ball $B\subseteq\RR^{d}$, and let $n=m^{(d-1)/(d+1)}$. Then a.a.s.\ $\xc(P)=\Theta(\sqrt{n})$.
\end{thm}

In \cref{thm:ball}, the significance of the expression defining $n$ is that the numbers of vertices and facets of the polytope $P$ are both a.a.s.\ of the form $\Theta(n)$ (see for example \cite{Rei05}). In the setting of \cref{thm:sphere}, the number of vertices of $P$ is always exactly $n$, and the expected number of facets of $P$ is known to be of the form $\Theta(n)$ (see \cite{BM85}).

We remark that the extension complexity of random polygons has previously been studied empirically in \cite{VGGT16}, using the model in \cref{thm:sphere} (with $d=2$). Furthermore, the $d=2$ case of \cref{thm:sphere} answers a question posed on the \emph{Open Problem Garden}~\cite{OPG}. On the other hand, the model of random polytopes in \cref{thm:ball} is more popular in probability theory\footnote{It is worth mentioning that the study of random polytopes is a classical topic in probability theory, started more than fifty years ago by R\'enyi and Sulanke~\cite{RS63}. See for example the surveys \cite{Bar07,Bar08,Gru97,Hug13,Sch97,Sch08,WW93} and the references therein.}.

The methods in the proofs of \cref{thm:sphere,thm:ball} seem to be quite robust, and actually do not use randomness in a very crucial way (basically, we only need the vertices of $P$ to be reasonably ``well-distributed'' and the facets of $P$ to be ``not too large''). We hope these ideas may have further applications in this area.

Our second contribution in this paper concerns the extension complexity of \emph{cyclic} polygons (i.e.\ of polygons whose vertices lie on a common circle). The proof of this result follows a similar overall approach as our proofs of \cref{thm:sphere,thm:ball}, but different ideas are required to make the approach work (in particular, we prove an inequality for slacks of cyclic polygons that plays an important role in our argument; see \cref{lem:circle-geometry}).

\begin{thm}
\label{thm:cyclic-polygon}Let $P$ be a cyclic polygon with $n$ vertices. Then $\xc(P)\leq 24\sqrt{n}$.
\end{thm}

The bound in \cref{thm:cyclic-polygon} is tight up to the constant factor $24$, as the $d=2$ case of \cref{thm:sphere} shows. We made no attempt to optimise the constant $24$.

It seems plausible that random polytopes exhibit nearly-worst-case extension complexity (for example, this was suggested as a possibility in \cite{Lee13,TheBlog,VGGT16}), and cyclic polygons seem to represent quite a diverse cross-section of the space of all polygons. So, in light of \cref{thm:sphere,thm:ball,thm:cyclic-polygon} it is tempting to (quite ambitiously) conjecture that for fixed $d$, in fact \emph{all} $d$-dimensional $n$-vertex polytopes have extension complexity $O(\sqrt n)$.

By contrast, until recently it seems that the prevailing belief in the area was that the dimension of a polytope $P$ provides essentially no control over its extension complexity. For example, there was a conjecture (see for example \cite{BL09}) that for all $n$ there is an $n$-gon with extension complexity $n$ (that is to say, the trivial bound cannot be improved even in dimension $d=2$). It has also been suggested (see for example \cite{Lee13,TheBlog}) that in contrast to the $d=2$ case of \cref{thm:sphere,thm:ball}, in fact almost all $n$-gons may have extension complexity $\Omega(n)$ (for some appropriate notion of ``almost all'').

In the last few years there have been a number of results challenging this belief, showing that in the case of dimension $d=2$ the extension complexity cannot be too large. First, Shitov~\cite{Shi14} and Padrol and Pfeifle~\cite{PP15} independently proved that every $n$-gon has extension complexity at most $(6n+6)/7$. Shortly afterwards, Shitov~\cite{Sh14v1} proved the much stronger result that every $n$-gon has extension complexity at most $o(n)$, and very recently he \cite{Sh14v2} improved this bound to $O(n^{2/3})$. However, it appears that before the present paper, no nontrivial upper bounds were known for any reasonably general class of polytopes in any dimension $d\ge 3$.

For our final theorem, we consider the case where the dimension $d$ is allowed to grow slowly with the number of vertices $n$. We show that in this case, the trivial upper bound $\xc(P)\le n$ is in fact nearly best-possible, confirming in a weak sense that the dimension of a polytope $P$ provides little control over its extension complexity.

\begin{thm}\label{thm:polytope-separation}
For any $n$, there is a polytope with at most $n$ vertices, dimension at most $n^{o(1)}$, and extension complexity $n^{1-o(1)}$.
\end{thm}

Prior to the present paper, it seems that the best known lower bound for the extension complexity in the setting of \cref{thm:polytope-separation} was $n^{\log_2(3/2)}$: Kaibel and Weltge~\cite{KW15} proved that the so-called \emph{correlation polytope} with $n=2^r$ vertices (which has dimension at most $r^2=n^{o(1)}$) has extension complexity at least $(3/2)^r$.

Using Yannakakis' theorem~\cite{Yan91}, we deduce \cref{thm:polytope-separation} from a slightly stronger result. Namely, answering a question of Hrube\v s~\cite{Hru12}, we prove that there is a nonnegative $n\times n$ matrix $M$ such that $\rank_+ M/\rank M=n^{1-o(1)}$. This matrix $M$ has a simple algebraic description, in a similar spirit as some matrices previously considered in connection with extension complexity and nonnegative rank (for example the so-called \emph{unique disjointness matrices}~\cite{FMPTdW15,VGGT16}, and the so-called \emph{Euclidean distance matrices}~\cite{BL09,Hru12,LC10,Shi19b}\footnote{Actually, in \cite{LC10} the authors claim that the Euclidean distance matrix of $n$ generic points in $\RR^1$ has rank $3$ and nonnegative rank $n$. This would imply the existence of an $n$-gon with extension complexity $n$. Unfortunately there is a fatal mistake in their proof.}). In order to lower-bound the nonnegative rank of $M$, we use a result of Sgall~\cite{Sga99} that can be interpreted as a two-family version of the celebrated Frankl--Wilson restricted intersection theorem~\cite{FW81}.

We prove \cref{thm:polytope-separation} in \cref{sec:separation} after introducing some basic definitions and notation in  \cref{sec:preliminaries-notation}. The rest of the paper is devoted to the proofs of \cref{thm:sphere,thm:ball,thm:cyclic-polygon}. In \cref{sect-outline}, we will give outlines of these proofs and explain the organisation of the rest of the paper.

\section{Preliminaries and Notation}
\label{sec:preliminaries-notation}

We say a matrix is \emph{nonnegative} if all of its entries are nonnegative.

\begin{defn}
 The \emph{nonnegative rank} $\rank_{+}M$ of a nonnegative $m\times n$ matrix $M\in\RR_{\geq 0}^{m\times n}$ is the minimum $r$ such that there is a factorisation $M=TU$, where
$T\in\RR_{\geq 0}^{m\times r}$ and $U\in\RR_{\geq 0}^{r\times n}$ are nonnegative matrices with $r$ columns and $r$ rows, respectively.
\end{defn}

One can equivalently define the nonnegative rank of $M\in\RR_{\geq 0}^{m\times n}$ to be the minimum $r$ such that the matrix $M$ can be written as the sum of $r$ nonnegative matrices of (ordinary) rank 1. Also note that $M$ has nonnegative rank at most $r$ if and only if there are nonnegative vectors $t_1,\dots,t_r\in \RR_{\geq 0}^m$ such that every column of $M$ can be written as a nonnegative linear combination of $t_1,\dots,t_r$. Similarly, $M$ has nonnegative rank at most $r$ if and only if there are nonnegative vectors $u_1,\dots,u_r\in \RR_{\geq 0}^n$ such that every row of $M$ can be written as a nonnegative linear combination of $u_1,\dots,u_r$. Finally, note that the nonnegative rank of a matrix is not affected by rescaling any of its rows of columns by a positive constant.

We can describe any polytope $P\su \RR^d$ by a finite list of linear constraints, i.e.\ we can represent $P$ in the form $P=\{ x\in \RR^d\mid Ax\le b\}$ for some matrix $A$ and some vector $b$ (so there is a constraint corresponding to each row of $A$). For a vertex $v$ of $P$ and a constraint $a_j\cdot x\leq b_j$ (where $a_j$ is a row of $A$, and $b_j$ is the corresponding entry of $b$), we say that $b_j-a_j\cdot v\geq 0$ is the \emph{slack} of $v$ with respect to this constraint.

\begin{defn}\label{def:slack-matrix}
A \emph{slack matrix} of a polytope $P \subseteq\RR^{d}$ is a matrix whose rows are indexed by the vertices of $P$ and whose columns are indexed by the linear constraints in some description $P=\{ x\in \RR^d\mid Ax\le b\}$, such that the entries of the matrix are precisely the slacks of the vertices with respect to the constraints.
\end{defn}

Note that a polytope $P$ does not uniquely define a slack matrix, because given any description of $P$ we can always rescale the constraints or add redundant constraints. Also, we will sometimes want to consider a description of a polytope $P$ which consists of equations as well as inequalities, i.e.\ we may wish to consider a description of the form $P=\{ x\in \RR^d\mid Ax\le b,\,A'x=b'\}$. We can still define a slack matrix in exactly the same way, using the inequalities in this description (as before, each column contains the slacks with respect to an inequality $a_j\cdot x\le b_j$). The equations $A'x=b'$ play no role in the slack matrix\footnote{Every equation can be equivalently expressed as two opposite inequalities, and the slack of every vertex with respect to these inequalities is zero. So, including these inequalities in the slack matrix would only introduce some additional zero columns, which would be inconsequential for our purposes.}.

For a $d$-dimensional polytope $P\su \RR^d$, we will usually consider a description $P=\{ x\in \RR^d\mid Ax\le b\}$ where the constraints are in correspondence with the facets of $P$. Then the columns of the corresponding slack matrix $M$ are indexed by the facets of $P$ (and, as always, the rows are indexed by the vertices of $P$). For a vertex $v$ and a facet $f$, the entry $M_{v,f}$ is the slack of the vertex $v$ with respect to the facet $f$ (or more precisely, with respect to the constraint associated with the facet $f$).

It is well known (see for example \cite[Theorem~14]{GGKPRT13}) that the rank of any slack matrix of a polytope $P\su \RR^d$ is 1 greater than the dimension of $P$ (note that the dimension of $P$ may be smaller than $d$). Also, note that the slack matrix of a polytope is always a nonnegative matrix. Yannakakis \cite{Yan91} proved the following famous theorem, connecting the extension complexity of a polytope to the nonnegative rank of its slack matrix\footnote{The theorem is stated in a slightly different way in Yannakakis' paper; see for example \cite[Theorem~2.6]{FKPT13} for this particular statement.}.

\begin{thm}\label{thm:yannakakis}
The extension complexity $\xc(P)$ of any polytope $P$ equals the nonnegative rank of any slack matrix of $P$.
\end{thm}

A consequence of Yannakakis' theorem is the following lemma. Roughly speaking, it states that upper bounds on the nonnegative rank of a matrix in terms of its (ordinary) rank are in a certain sense equivalent to upper bounds on the extension complexity of a polytope in terms of its dimension. Although we were not able to find this particular statement in the literature, very similar facts have implicitly been proved in various papers (see for example \cite[Theorem~3.1]{Shi14}).

\begin{lem}\label{lem:gap-equivalence}
For $n,d\in \NN$, let $f_{\mathrm{xc}}(n,d)$ be the maximum extension complexity of a polytope with dimension at most $d$ and at most $n$ facets. Furthermore, for $n,r\in \NN$, let $f_{+}(n,r)$ be the maximum nonnegative rank of a nonnegative matrix with at most $n$ columns and rank at most $r$. Then for any $n\in \NN$ and $r\ge2$, we have $f_{\mathrm{xc}}(n,r-1)=f_{+}(n,r)$.
\end{lem}

\begin{proof}
Consider a polytope $P$ with dimension at most $r-1$, at most
$n$ facets and extension complexity $f_{\mathrm{xc}}(n,r-1)$.
The slack matrix $M$ of $P$ has at most $n$ columns and rank at most $r$. By Yannakakis' theorem we have $\rank_{+}M=f_{\mathrm{xc}}(n,r-1)$, 
which shows that $f_{+}(n,r)\ge f_{\mathrm{xc}}(n,r-1)$.

On the other hand, consider a nonnegative matrix $M$ with at most $n$ columns, with rank at most $r$, and with nonnegative rank $f_{+}(n,r)$.
We may assume that $M$ has exactly $n$ columns (otherwise we can add a suitable number of all-zero columns),
we may assume that $M$ does not have any row consisting entirely of zeros (otherwise we can omit any such row), and we may also assume that the entries in each row of $M$ sum to $1$ (rescaling the rows of $M$ does not affect its nonnegative rank).
Then every row of $M$ is a vector lying in the standard simplex
\[
\Delta=\{(x_{1},\dots,x_{n})\in\RR^{n}\mid x_{1}+\dots+x_n=1,\,x_{i}\ge0\text{ for each }i\}.
\]
Since $M$ has rank at most $r$, the rows of $M$ lie in an affine subspace of $\RR^{n}$ with dimension at most $r-1$. Let $P$ be the polytope with dimension at most $r-1$ obtained by intersecting this affine subspace with the simplex $\Delta$. This polytope $P\su \RR^n$ can be described by various equations and the inequalities $x_i\ge 0$ (or, equivalently $-x_i\leq 0$, to be consistent with our notation in \cref{def:slack-matrix}) for $i=1,\dots,n$. Therefore $P$ has at most $n$ facets.

Now, let $M'$ be the slack matrix of $P$ corresponding to this description (i.e.\ the $i$-th column of $M'$ contains the slacks of the vertices of $P$ with respect to the constraint $x_i\ge 0$). By Yannakakis' theorem, $\rank_+ M'=\xc(P) \le f_{\mathrm{xc}}(n,r-1)$.

For each $i=1,\dots,n$ and each vertex $v$, the slack of $v$ with respect to the constraint $x_i\ge 0$ is precisely the $i$-th coordinate of $v$. This means that the rows of $M'$ are precisely the coordinate vectors of the vertices of $P$. Now, each row of the matrix $M$ is the coordinate vector of a point in $P$, and therefore is a convex combination of the vertices of $P$. That is to say, each row of $M$ can be written as a convex combination of the rows of $M'$. Therefore we have $f_{+}(n,r)=\rank_+M\le \rank_+ M'=\xc(P) \le f_{\mathrm{xc}}(n,r-1)$.
\end{proof}

\subsection{Notation}

In this paper, we use the notation $\NN=\{1,2,\dots\}$ for the positive integers. All logarithms are to base $e$, unless otherwise specified.

For most of the paper (from \cref{sect-outline} onward), $d\geq 2$ will always be a fixed dimension. We denote by $B\su \RR^d$ the $d$-dimensional unit ball around the origin, and by $S\su \RR^d$ the $(d-1)$-dimensional unit sphere around the origin (i.e.\ the boundary of the ball $B$).

We use common asymptotic notation. Let us stress that in all of these asymptotic notations the variable $d$ will be treated as fixed (i.e.\ the implicit constants are allowed to depend on $d$). For real-valued functions $f$ and $g$ (which will usually, but not always, be functions of $n\in \NN$), we write $f=O(g)$ to mean that there is some constant $C>0$ such that $\vert f\vert \leq Cg$. If $g$ is nonnegative, we write $f=\Omega(g)$ to mean that there is $c>0$ such that $f \geq cg$ (if $f$ and $g$ are functions of $n\in \NN$, we only require $f(n) \geq cg(n)$ for sufficiently large $n$). If $g$ is nonnegative, we furthermore write $f=\Theta(g)$ if $f=O(g)$ and $f=\Omega(g)$, i.e.\ if there are constants $c>0$ and $C>0$ such that $c g\leq f \leq C g$. For functions $f:\NN\to \RR$ and $g:\NN\to \RR_{>0}$, we write $f=o(g)$ if $f(n)/g(n)\to 0$ as $n\to\infty$, and we write $f=\omega(g)$ if $f(n)/g(n)\to \infty$ as $n\to\infty$.

\section{Separation between rank and nonnegative rank}
\label{sec:separation}

Answering a question of Hrube\v s~\cite[Question~1]{Hru12}, we prove that there exists a matrix with near-optimal separation between rank and nonnegative rank.
\begin{thm}
\label{thm:separation}For every $n\in \mathbb{N}$, there is a nonnegative $n\times n$ matrix $M$
satisfying $\rank_{+}M/\rank M=n^{1-o(1)}$.
\end{thm}

To give a more precise estimate for the $o(1)$-term in \cref{thm:separation}, our proof shows that one can choose the matrix $M$ such that $\rank_{+}M/\rank M\geq n^{1-O\left(\log\log n/\sqrt{\log n}\right)}$.

\cref{thm:polytope-separation} stated in the introduction follows easily from \cref{thm:separation} using \cref{lem:gap-equivalence}. Indeed, \cref{thm:separation} implies that for every $n\in \NN$ there exists a nonnegative matrix with $n$ columns, rank $n^{o(1)}$ and nonnegative rank $n^{1-o(1)}$. By \cref{lem:gap-equivalence}, this means that there exists a polytope of dimension $n^{o(1)}$ with at most $n$ facets and extension complexity $n^{1-o(1)}$. Considering the polar dual of this polytope, we obtain a polytope of dimension $n^{o(1)}$ with at most $n$ vertices and extension complexity $n^{1-o(1)}$. This proves \cref{thm:polytope-separation}.

To prove \cref{thm:separation} we will need the well-known \emph{rectangle covering bound} for the nonnegative rank of a matrix. Given a nonnegtaive matrix $M$ with rows indexed by some finite set $I$ and columns indexed by some finite set $J$, a \emph{rectangle} is a product $\mathcal R=I'\times J'$ for some subsets $I'\su I$ and $J'\su J$. A \emph{rectangle covering} of the matrix $M$ is a collection of (possibly overlapping) rectangles $\mathcal R_1,\dots,\mathcal R_k$ such that we have $\{(i,j)\in I\times J\mid M_{i,j}>0\}=\mathcal R_1\cup \dots \cup\mathcal R_k$ (in other words, such that the support of $M$ is the union of the rectangles $\mathcal R_1,\dots,\mathcal R_k$, which in particular means that the matrix $M$ is strictly positive on all of these rectangles). The \emph{rectangle covering number} $\operatorname{rc}(M)$ of the matrix $M$ is the smallest possible number of rectangles in a rectangle covering of $M$. This parameter is also known as the \emph{Boolean rank} of the support matrix of $M$. The following bound is well-known and easy to prove, see for example \cite[Equation~(2)]{FKPT13}.

\begin{fact}
For any nonnegative matrix $M$, we have $\operatorname{rc}(M)\le \rank_+M$.
\end{fact}

We will also use the following theorem due to Sgall, appearing as \cite[Corollary~3.5]{Sga99}\footnote{We remark that there is an (inconsequential) typo in the statement of this result. The statement as printed in \cite[Corollary~3.5]{Sga99} is that $|\mathcal{A}|\cdot|\mathcal{B}|\le 2^{r+s-1}\cdot \binom{r}{\le s-1}\le 2^{r+s+H(s/r)r}$, but the middle term should be corrected to $2^{r+s-1}\cdot \binom{r}{s-1}$.}. Let $H:(0,1)\to \RR_{\geq 0}$ be the binary entropy function given by $H(x)=-x\log_{2}x-(1-x)\log_{2}(1-x)$.
\begin{thm}
\label{lem:sgall}Let $1\leq s<m\leq r$ be integers. Let $\mathcal{A}$ and $\mathcal{B}$ be families of subsets
of $\{ 1,\dots,r\} $, and suppose that the intersection sizes
$|A\cap B|$ for $A\in\mathcal{A}$ and $B\in\mathcal{B}$
take only $s$ different values modulo $m$.
Then
\[
|\mathcal{A}|\cdot|\mathcal{B}|\le2^{r+s+H(s/r)r}.
\]
\end{thm}

Finally, we will need the following simple fact.

\begin{lem}\label{lem:matrix-support}
Let $r\geq 1$ and $m\geq 2$ integers. Then the probability that two independent uniformly random vectors $a,b\in \{0,1\}^r$ satisfy $a\cdot b\equiv 0\pmod{m}$ is at most $3/4$.
\end{lem}
\begin{proof}
We write $a=(a_1,\dots,a_r)$ and $b=(b_1,\dots,b_r)$. Let us condition on the outcome of $a_1,\dots,a_{r-1}$ and $b_1,\dots,b_{r-1}$. Then $a_r b_r$ is equal to one with probability $1/4$ and is equal to zero with probability $3/4$. At most one of these outcomes will satisfy $a_1b_1+\dots+a_rb_r\equiv 0\pmod{m}$, so the probability of having $a\cdot b\equiv 0\pmod{m}$ is at most $3/4$.
\end{proof}

We are now ready to prove \cref{thm:separation}. The basic idea behind the proof is to construct a suitably chosen low-rank matrix $M$ whose rows and columns are indexed by vectors in $\{0,1\}^r$, and to use \cref{lem:sgall} to show that the matrix $M$ does not have any large rectangles containing only positive entries. Then the rectangle covering bound will show that the matrix $M$ has high nonnegative rank.

\begin{proof}[Proof of \cref{thm:separation}]
We may assume that $n$ is a power of $2$. Indeed, if we can construct an appropriate $n'\times n'$ matrix $M'$ for $n'=2^{\floor{\log_2 n}}$, we can obtain an $n\times n$ matrix $M$ by adding additional all-zero rows and all-zero columns to $M'$. We then have $\rank_+ M/\rank M=\rank_+ M'/\rank M'=(n')^{1-o(1)}=n^{1-o(1)}$.

So let us assume that $r=\log_{2}n$ is an integer, and that $r\geq 4$. Furthermore let $m=\lceil \sqrt{r}\,\rceil$, and let $Q=m\mathbb Z\cap \{ 0,1,\dots,r\}$ 
be the set of all multiples of $m$ in $\{ 0,1,\dots,r\}$. Note that $\vert Q\vert \leq \sqrt{r}+1=O(\sqrt{r})$.

We now define $M$ to be the $n\times n$ matrix with rows and columns indexed by $\{ 0,1\} ^{r}$, where for any $a,b\in \{ 0,1\} ^{r}$ we let $M_{a,b}=\prod_{q\in Q}(a\cdot b-q)^{2}$. Clearly, all entries of the matrix $M$ are nonnegative.

We also claim that $\rank M =r^{O(\sqrt{r})}$. Indeed, writing $a=(a_1,\dots,a_r)\in \{ 0,1\} ^{r}$ and $b=(b_1,\dots,b_r)\in \{ 0,1\} ^{r}$, we can multiply out the definition $M_{a,b}=\prod_{q\in Q}(a\cdot b-q)^{2}$ and obtain a representation of $M_{a,b}=f(a_1b_1,\dots,a_rb_r)$ as an $r$-variable polynomial of degree $2\left|Q\right|$ in the terms $a_1b_1,\dots,a_rb_r$. The polynomial $f$ consists of at most $(r+1)^{2\left|Q\right|}$ monomials, and splitting $f$ into monomials gives rise to a representation of $M$ as a sum of rank-$1$ matrices. Thus, we obtain $\rank M\le (r+1)^{2\left|Q\right|}=r^{O(\sqrt{r})}$ as desired.

By construction, for any $a,b\in \{0,1\}^r$, the matrix entry $M_{a,b}$ is zero precisely when $a\cdot b$ is divisible by $m$. Hence, by \cref{lem:matrix-support} at least a quarter of the $n^2$ entries of the matrix $M$ are nonzero.

Now, suppose $\mathcal{A}\times\mathcal{B}$ is a rectangle, given by a family of vectors $\mathcal{A}\su \{0,1\}^r$ and a family of vectors $\mathcal{B}\su \{0,1\}^r$, such that $M_{a,b}>0$ for all $a\in \mathcal{A}$ and $b\in \mathcal{B}$. We can then interpret $\mathcal{A}$ and $\mathcal{B}$ as families of subsets of $\{ 1,\dots,r\}$, and note that for any $A\in \mathcal{A}$ and $B\in \mathcal{A}$ we have $\vert A\cap B\vert\not\equiv 0\pmod{m}$. Indeed, if $a,b\in \{0,1\}^r$ are the indicator vectors corresponding to $A$ and $B$, we have $\vert A\cap B\vert=a\cdot b\not\equiv 0\pmod{m}$, since $M_{a,b}>0$. Thus, the intersection sizes $|A\cap B|$ for $A\in\mathcal{A}$ and $B\in\mathcal{B}$ take at most $m-1\leq \sqrt{r}$ different values modulo $m$, and  \cref{lem:sgall} implies that
\[|\mathcal{A}|\cdot|\mathcal{B}|\le 2^{r+\sqrt{r}+H\left(\sqrt{r}/r\right)r}=2^{r+O\left(\sqrt{r}\log r\right)}\]
(for the first inequality here we used the fact that the binary entropy function $H:(0,1)\to \RR_{\geq 0}$ is increasing on the interval $(0,\frac{1}{2})$, and that $\sqrt{r}/r<\frac{1}{2}$ by our assumption that $r\geq 4$).

In other words, any rectangle in the support of the matrix $M$ consists of at most $2^{r+O(\sqrt{r}\log r)}$ entries. Since the support of  $M$ consists of at least $n^2/4$ entries, at least
\[\frac{n^2/4}{2^{r+O\left(\sqrt{r}\log r\right)}}=2^{2r-2-r-O\left(\sqrt{r}\log r\right)}=2^{r-O\left(\sqrt{r}\log r\right)}\]
rectangles are needed to cover the support of the matrix $M$. By the rectangle covering bound, it follows that $\rank_+ M\ge 2^{r-O(\sqrt{r}\log r)}$. All in all, we obtain
\[\frac{\rank_+ M}{\rank M}\geq \frac{2^{r-O(\sqrt{r}\log r)}}{r^{O(\sqrt{r})}}=2^{r-O\left(\sqrt{r}\log r\right)}=n^{1-O\left(\log{r}/\sqrt{r}\right)}=n^{1-O\left(\log\log n/\sqrt{\log n}\right)}=n^{1-o(1)},\]
as desired.
\end{proof}

\section{Proof outlines for \texorpdfstring{\cref{thm:sphere,thm:ball,thm:cyclic-polygon}}{Theorems \ref{thm:sphere} and \ref{thm:cyclic-polygon}}}
\label{sect-outline}

The rest of the paper is devoted to the proofs of \cref{thm:sphere,thm:ball,thm:cyclic-polygon}, and this section contains outlines of these proofs. At the end of the section, we describe how the components of these proofs are organised in the rest of the paper.

\subsection{Random polytopes}

The lower bounds on the extension complexity in \cref{thm:sphere,thm:ball} can be proved with an approach of Fiorini, Rothvo{\ss} and Tiwary~\cite{FRT12} (their work was in the case of $d=2$, but the approach can easily be generalised to higher dimensions). We will now outline the proofs of the upper bounds in \cref{thm:sphere,thm:ball}, which require several new ideas.

In the setting of \cref{thm:ball}, where $P$ is the convex hull of many random points inside the ball $B$, intuition suggests that $P$ is likely to ``fill out'' most of the ball, meaning that its vertices are likely to be very close to the surface of $B$. This intuition can be made precise, and for this reason the proofs of the upper bounds in \cref{thm:sphere,thm:ball} are very similar. We will therefore focus this outline on the setting of \cref{thm:sphere}, where each of our random points on the sphere is automatically a vertex of the polytope $P$. So let $P$ be a random polytope, given as the convex hull of a random set $V$ of $n$ vertices on the sphere, let $F$ be the set of facets of $P$, and let $M$ be a slack matrix of $P$ (with rows indexed by $V$ and columns indexed by $F$).

Perhaps the most important insight driving the proof is as follows.
Given our polytope $P$, consider a small ``patch'' of facets $F'$
(near the north pole of the sphere, say), and consider a collection
of vertices $V'$ which are far away from $F'$ (being at least
five times as far from the north pole as the facets in $F'$, say; see the left side of \cref{fig:lampshade}). Then,
if we consider the $V'\times F'$ submatrix\footnote{We remark that this can be interpreted as the slack matrix for a \emph{pair} of polyhedra. The idea of considering pairs of polyhedra can also be found for instance in \cite{Pas12}.} $M[V',F']$
of the slack matrix $M$ consisting only of the slacks between vertices
in $V'$ and facets in $F'$, we have $\rank_{+}M[V',F']=O(1)$.

The reason for this is that we can ``hang a polyhedral lampshade''
from the gap between $F'$ and $V'$ (see the right side of \cref{fig:lampshade}). More precisely, we can find a polytope $Q$ with relatively few vertices (looking like
a polyhedral approximation of a truncated cone) which fully encloses
all the vertices in $V'$, and which lies completely on the ``positive slack'' side of the facets in $F'$ (meaning that for each facet $f\in F'$, $Q$ and $P$ lie on the same side of the hyperplane through $f$) . Crucially, one can ensure that the number
of vertices of $Q$ is bounded, depending on
the ambient dimension $d$ (but not on $|F'|$ or $|V'|$). Every vertex in $V'$ can then be expressed as a convex combination of the $O(1)$ vertices of the polytope $Q$. Now,
for a facet $f$ corresponding to a constraint $a\cdot x\le b$, the
slack function $\psi_{f}:x\mapsto b-a\cdot x$ is an affine-linear map. So for
each vertex $w$ of $Q$ we can consider the nonnegative vector of slacks $u_w=(\psi_{f}(w))_{f\in F'}$,
and observe that every row of the matrix $M[V',F']$ can be expressed
as a convex combination of these vectors $u_w$. This certifies that
$\rank_{+}M[V',F']=O(1)$.

\begin{figure}
\begin{center}\includegraphics[scale=0.7,trim={2.5cm 2cm 2.5cm 4cm},clip]{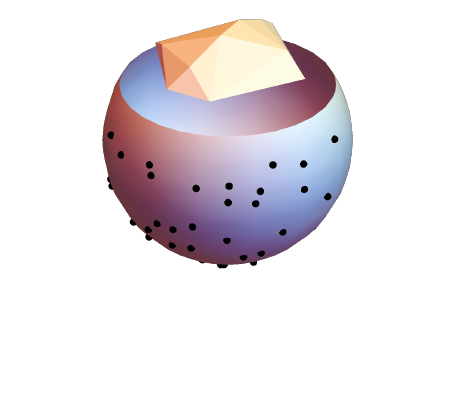}\includegraphics[scale=0.7,trim={0 2cm 0 4cm},clip]{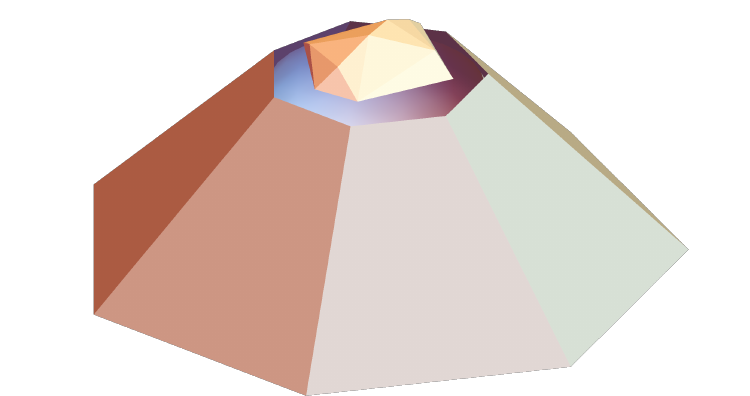}
\end{center}\caption{\label{fig:lampshade}On the left, a small patch of facets near the north pole is far away
from a collection of vertices. On the right, a ``polyhedral lampshade''
encloses all the vertices in our collection, and lies entirely on the ``positive slack'' side of each of the facets in our patch.}
\end{figure}

In fact, one can use an appropriately chosen ``polyhedral lampshade'' as above not only to show that the $V'\times F'$ submatrix $M[V',F']$ of the slack matrix $M$ satisfies $\rank_{+}M[V',F']=O(1)$, but also to show this for certain modified versions of the matrix $M[V',F']$ (where we are allowed to make certain subtractions from $M[V',F']$).

This approach is heavily inspired by Shitov's proof that every polygon has sublinear extension complexity~\cite{Sh14v1}\footnote{To clear up some potential confusion: Shitov's original $o(n)$ bound for the extension complexity of any $n$-gon~\cite{Sh14v1} and his later $O(n^{2/3})$ bound~\cite{Sh14v2} appeared on the arXiv as multiple versions of the same paper. Since these two versions feature completely different proofs, and we want to refer specifically to a lemma in the first version, we have made the slightly unusual choice to cite them as different papers.}. In fact, in the special case where $d=2$ this argument essentially appears in \cite[Lemma~3.1]{Sh14v1}. However, there are several difficulties in higher dimensions that do not present themselves in the two-dimensional case: in particular, we remark that when $d=2$ it is actually not necessary to have a ``gap'' separating the facets in $F'$ from the vertices in $V'$.

In order to find ``patches'' of facets to apply the above ideas in our proof of \cref{thm:sphere}, we consider a suitably chosen collection of $O(\sqrt{n})$ spherical caps covering the surface of the sphere, each with the same radius $\eps$ (chosen such that the surface area of each cap is about $1/\sqrt n$). Recalling that $P$ is a random polytope, it is easy to show that its facets are typically quite small, and we will be able to show that a.a.s.\ each facet of $P$ is ``inside'' one of the caps in our collection (for a slightly technical notion of being ``inside'', which is not too important for this proof outline). Using the randomness of $P$ it is also easy to show that a.a.s.\ each of the caps in our collection contains $O(\sqrt{n})$ vertices of $P$ (in fact, we will need something slightly stronger, namely that each cap has at most $O(\sqrt{n})$ vertices of $P$ within distance $5\eps$ of the centre of the cap).

We then want to assign colours to each of our caps, in such a way that any two caps of the same colour are far apart from each other (say, their centres have distance at least $30\eps$). Using standard packing-and-covering arguments, we can choose our collection of $O(\sqrt{n})$ caps in such a way that only $O(1)$ colours are required. For each colour $c$, let $F_c\su F$ be the set of facets lying ``inside'' a cap of colour $c$. It will then suffice to show that the $V\times F_c$ submatrix $M[V,F_c]$ of the slack matrix $M$ (consisting only of the columns corresponding to facets in $F_c$) has nonnegative rank $O(\sqrt{n})$. Indeed, showing $\rank_+ M[V,F_c]=O(\sqrt{n})$ for each of the $O(1)$ colours $c$ would imply that $\xc(P)=\rank_+ M=O(\sqrt{n})$. The left side of Figure \ref{figure-slack-matrix-decomposition} shows the decomposition of $M$ into the submatrices $M[V,F_c]$.

\begin{figure}
    \centering
\begin{tikzpicture}[scale=0.75]
\fill[gray,opacity=0.3] (-2,2.25) rectangle (-1.25,1.35);
\fill[gray,opacity=0.3] (-1.25,1.35) rectangle (-0.4,0.4);
\fill[gray,opacity=0.3] (-0.4,0.4) rectangle (0.6,-0.4);
\fill[gray,opacity=0.3] (0.6,-0.4) rectangle (1.3,-1.45);
\fill[gray,opacity=0.3] (1.3,-1.45) rectangle (2,-2.25);
\draw  (-2,-2.25) rectangle (2,2.25);
\draw  (-2,1.35) -- (2,1.35);
\draw  (-2,0.4) -- (2,0.4);
\draw  (-2,-0.4) -- (2,-0.4);
\draw  (-2,-1.45) -- (2,-1.45);
\draw  (-1.25,-2.25) -- (-1.25,2.25);
\draw  (-0.4,-2.25) -- (-0.4,2.25);
\draw  (0.6,-2.25) -- (0.6,2.25);
\draw  (1.3,-2.25) -- (1.3,2.25);
\node[anchor=south] at (-0.8,2.25) {$F^D$};
\node[anchor=east] at (-2,0.875) {$V^D$};
\draw [decorate, decoration = {brace,raise=15pt}] (-2,2.25) --  (2,2.25) node[pos=0.5, above=15pt] {$F_c$};
\draw [decorate, decoration = {brace,raise=20pt}] (-2,-2.25) --  (-2,2.25) node[pos=0.5, left=21pt] {$V_c$};
\draw[dotted, gray]  (-8.625,0) -- (-8.625,-2.25);
\draw[dotted,gray]  (-8.2,0) -- (-8.2,-2.25);
\draw[dotted,gray]  (-7.7,0) -- (-7.7,-2.25);
\draw[dotted,gray]  (-7.35,0) -- (-7.35,-2.25);
\draw  (-9,-2.25) rectangle (-7,2.5);
\draw  (-9,0) -- (-7,0);
\draw  (-8.625,0) -- (-8.625,2.5);
\draw  (-8.2,0) -- (-8.2,2.5);
\draw  (-7.7,0) -- (-7.7,2.5);
\draw  (-7.35,0) -- (-7.35,2.5);
\node[anchor=south] at (-8,2.5) {$F_c$};
\node[anchor=east] at (-9,1.25) {$W_c$};
\node[anchor=east] at (-9,-1.125) {$V_c$};
\draw [decorate, decoration = {brace,raise=20pt}] (-9,-2.25) --  (-9,2.5) node[pos=0.5, left=21pt] {$V$};
\draw  (-17,-1.25) rectangle (-13.5,1.25);
\draw  (-16.1,-1.25) -- (-16.1,1.25);
\draw  (-14.5526,-1.25) -- (-14.5526,1.25);
\node at (-15.3,0) {$\dots$};
\node[anchor=east] at (-17,0) {$V$};
\node[anchor=south] at (-14,1.25) {$F_c$};
\draw [decorate, decoration = {brace,raise=15pt}] (-17,1.25) --  (-13.5,1.25) node[pos=0.5, above=15pt] {$F$};
\draw [->,>=stealth,very thick] (-10.9,.0) -- (-14.1,0);
\draw [->,>=stealth,very thick] (-3.9,.0) -- (-8,-1.125);
\end{tikzpicture}
\vspace{0.2cm}
\caption{The left side shows the slack matrix of $P$, divided into $O(1)$ submatrices $M[V,F_c]$ for each colour $c$. The middle picture shows one of these matrices $M[V,F_c]$, with its row set divided into $W_c$ (the set of vertices far away from all caps of colour $c$) and $V_c=V\setminus W_c$, and its column set divided into $O(\sqrt{n})$ blocks corresponding to the caps of colour $c$. Each of the submatrices in the top part of the middle picture has nonnegative rank $O(1)$ by the lampshade argument. The right side shows the matrix $M[V_c,F_c]$, with the diagonal blocks $M[V^D,F^D]$ in grey.}
\label{figure-slack-matrix-decomposition}
\end{figure}
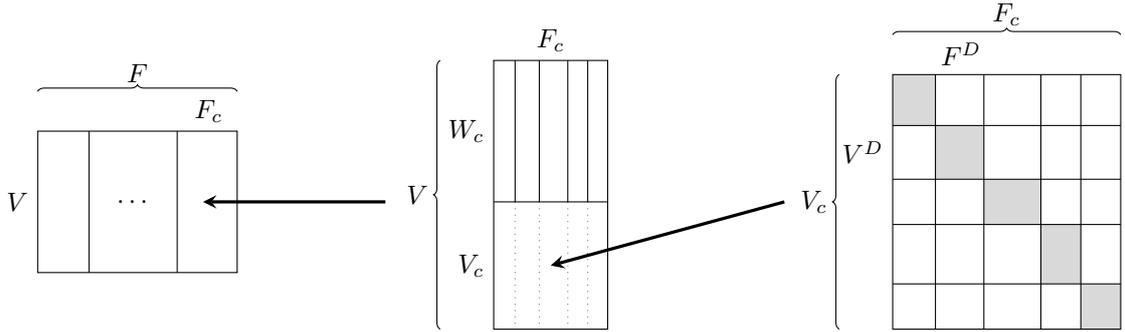

For any of the $O(\sqrt{n})$ caps of colour $c$ we obtain a patch of facets of $P$, namely the facets ``inside'' this cap. When applying the lampshade argument to such a patch of facets $F'$ inside a given cap of colour $c$, we obtain that $\rank_+ M[V',F']=O(1)$ for any set $V'$ of vertices that are sufficiently far away from the cap (say, that have distance at least $5\eps$ from the centre of the cap). We can use this argument to show that $\rank_+ M[W_c,F_c]\leq \sum_{F'} \rank_+ M[W_c,F']=O(\sqrt{n})$, where $W_c$ is the set of vertices that are far away from all caps of colour $c$ (here, the sum is over the patches of facets $F'$ obtained from each of the $O(\sqrt{n})$ caps of colour $c$). The middle picture in Figure \ref{figure-slack-matrix-decomposition} shows the matrix $M[V,F_c]$, with its top part $M[W_c,F_c]$ being decomposed into the submatrices $M[W_c,F']$ for these patches of facets $F'$ (each of which satisfies $\rank_+ M[W_c,F']=O(1)$).

It remains to show that $\rank_+ M[V_c,F_c]\leq O(\sqrt{n})$, where $V_c=V\setminus W_c$ is the set of vertices of $P$ that are close to some cap of colour $c$. We will use the lampshade argument once again, but this time we will have to use it in its more general form, bounding the nonnegative rank of a matrix obtained by making certain subtractions from $M[V_c,F_c]$.

The details of this last part of the proof are a bit technical, but to give some rough intuition it is helpful to think about the structure of $M[V_c,F_c]$. As before, we can partition $F_c$ into patches of facets $F^D$, for the different caps $D$ of colour $c$ (where $F^D$ is the set of facets inside the cap $D$). The caps actually also provide a natural partition of $V_c$ into sets $V^D$, where each set $V^D$ is the set of vertices that are close to the cap $D$ (i.e.\ that have distance at most $5\eps$ from the centre of $D$). Now, these partitions of $V_c$ and $F_c$ induce a partition of the matrix $M[V_c,F_c]$ into blocks $M[V^D,F^{D'}]$, each containing the slacks between the vertices close to some cap $D$ of colour $c$ and the facets inside some cap $D'$ of colour $c$. Since the caps of colour $c$ are very far apart from each other, the entries in the ``diagonal'' blocks $M[V^D,F^{D}]$ are much smaller than the entries in the non-diagonal blocks $M[V^D,F^{D'}]$ for $D\neq D'$. The right side of Figure \ref{figure-slack-matrix-decomposition} shows this partition of the matrix $M[V_c,F_c]$, with the diagonal blocks $M[V^D,F^{D}]$ coloured grey.

If we imagine for a moment that the entries in the diagonal blocks $M[V^D,F^D]$ were not just small but were in fact zero, then we would be in a position to apply the lampshade argument: For each cap $D$ of colour $c$, the matrix $M[V_c,F^D]$ would consist of the zero block $M[V^D,F^D]$ and the matrix $M[V_c\setminus V^D,F^D]$, where all vertices in $V_c\setminus V^D$ are far away from $D$ (and $F^D$ is a patch of facets inside the cap $D$). We would then be able to apply the lampshade argument with $F'=F^D$ and $V'=V_c\setminus V^D$ to show that $\rank_+M[V_c,F^D]=\rank_+M[V_c\setminus V^D,F^D]=O(1)$, which would imply that $\rank_+M[V_c,F^c]\leq \sum_D \rank_+M[V_c,F^D]=O(\sqrt n)$.

Of course, we cannot assume the diagonal blocks $M[V^D,F^D]$ contain only zeroes. However, it turns out that we can make certain subtractions from the matrix $M[V_c,F_c]$, and then apply the aforementioned strategy to the $V_c\times F_c$ matrix $K$ resulting from these subtractions. The aim of these subtractions is to make all entries in the diagonal blocks $K[V^D,F^D]$ of the matrix $K$ zero, such that the arguments from the previous paragraph can be applied to $K$ (with a suitable generalisation of the lampshade argument), implying $\rank_+ K=O(\sqrt n)$. 

For making the subtractions, we will define a collection of $O(\sqrt{n})$ nonnegative vectors, constructed in a certain way from the entries of the diagonal blocks $M[V^D,F^D]$ of the matrix $M[V_c,F_c]$. From each row of the original matrix $M[V_c,F_c]$ we will subtract one of these vectors, in such a way that the resulting matrix $K$ is nonnegative and all of its diagonal blocks $K[V^D,F^D]$ are indeed zero. In order to achieve this with a collection of only $O(\sqrt n)$ vectors, we will use the fact that for each cap of colour $c$ there are only $O(\sqrt{n})$ vertices within distance $5\eps$ of the centre of the cap. Since we only subtracted $O(\sqrt{n})$ different vectors, we obtain that $\rank_+M[V_c,F_c]\leq \rank_+ K+O(\sqrt{n})\leq O(\sqrt{n})$, as desired.

\subsection{Cyclic polygons}

The overall approach for our proof of \cref{thm:cyclic-polygon} is similar to the proofs of \cref{thm:sphere,thm:ball} outlined above. Let $P$ be a cyclic polygon (with $n$ vertices on the unit circle), let $V$ be its set of vertices, $F$ its set of facets, and $M$ its slack matrix. We will consider a collection of $O(\sqrt{n})$ arcs on the unit circle, similar to the collection of caps considered above in the outline of the proof of \cref{thm:sphere}.

In contrast to the previous subsection, where we had a random polytope $P$ whose vertices were typically very well-distributed over the sphere, in the present setting the vertices might be very clustered in certain places. Since we want each of our arcs to contain only $O(\sqrt{n})$ vertices of $P$, we can no longer choose all of the arcs to be of the same size. We therefore need a more general notion of what it means for two arcs to be far away from each other, when the arcs are of different lengths (recall that for the setting of \cref{thm:sphere} above we considered two caps of radius $\eps$ to be far apart if their centres have distance at least $30\eps$). We say that two arcs of the unit circle of lengths $\eps$ and $\eps'$ are ``well-separated'' if they have distance at least, say, $5\min\{\eps, \eps'\}$ from each other.

For any cyclic polygon $P$, we can divide its facets (edges) into $\sqrt{n}$ consecutive blocks. In this way, we obtain a collection of $O(\sqrt{n})$ arcs of the circle, each containing $O(\sqrt{n})$ vertices of $P$, and such that each facet of $P$ is ``inside'' exactly one of these arcs. The arcs may have very different lengths from each other, but it is not hard to show that we can still colour the arcs with $O(1)$ colours, such that any two arcs of the same colour are well-separated from each other (in the sense defined above).

We would like to more or less imitate the proof of \cref{thm:sphere} in this setting. The main problem is that if the arcs have different sizes, then the entries of some of the diagonal blocks $M[V^D,F^D]$ can be much larger than the entries in some of the non-diagonal blocks $M[V^D,F^{D'}]$. Therefore, if we try to naively perform the same subtractions from the matrix $M[V_c,F_c]$ as we did in the proof of \cref{thm:sphere}, then the resulting matrix $K$ may have negative entries (in which case its nonnegative rank is undefined or infinite, depending on one's convention). In order to overcome this problem, we will first rescale the rows of $M[V_c,F_c]$ before performing any subtractions. It turns out that we can construct suitable rescaling factors inductively, taking advantage of certain geometric properties of the circle.

We remark that instead of the higher-dimensional ``lampshade'' argument mentioned in the previous subsection, here it is convenient to use the original two-dimensional lemma of Shitov \cite[Lemma~3.1]{Sh14v1} that inspired our higher-dimensional version. Shitov used this lemma to show that every $n$-vertex polygon satisfying a certain ``admissibility'' condition has extension complexity $O(\sqrt{n})$, and he in turn used this to show that any $n$-vertex polygon has extension complexity $o(n)$. Shitov's proof that ``admissible'' polygons have extension complexity $O(\sqrt{n})$ can be interpreted in a way that resembles the idea of rescaling the rows of the matrix $M[V_c,F_c]$. However, his rescaling factors are given by explicit formulas, and this approach crucially relies on the ``admissibility'' of the polygon. In our setting Shitov's ``admissibility'' condition does not hold, and we therefore developed a completely different way to find suitable rescaling factors for the rows of the matrix $M[V_c,F_c]$.

\subsection{Organisation of the rest of the paper}

\cref{sec:lower-bound,sec:basic-lemmas,sec:lampshade,sec:upper-bound,sect-cone} will be devoted to proving \cref{thm:ball,thm:sphere} on the extension complexity of random polytopes. First, in \cref{sec:lower-bound} we prove the lower bounds. \cref{sec:basic-lemmas} contains some basic lemmas used in the proof of the upper bound. In particular, this section contains some lemmas about spherical caps and about properties of random polytopes. \cref{sec:lampshade} features the key lemma stating that certain (modified versions of) submatrices of the slack matrix have bounded nonnegative rank. This lemma is what is referred to as the ``lampshade argument'' above, since its proof relies on constructing a suitable ``lampshade'' as depicted on the right side of \cref{fig:lampshade}. The actual geometric construction of this ``lampshade'' is deferred to \cref{sect-cone}. Relying on this key lemma, in \cref{sec:upper-bound} we prove the upper bound in \cref{thm:ball,thm:sphere}.

In \cref{sec:preparations-cyclic,sec:cyclic,section-geometry-circle}, we prove \cref{thm:cyclic-polygon}, upper-bounding the extension complexity of cyclic polygons. More specifically, \cref{sec:preparations-cyclic} contains some preparations, while \cref{sec:cyclic} contains the actual proof of \cref{thm:cyclic-polygon} (and \cref{section-geometry-circle} contains the proof of a geometric lemma used in the proof of \cref{thm:cyclic-polygon}).

\section{Lower-bounding the extension complexity of random polytopes}
\label{sec:lower-bound}

We will deduce the lower bounds in \cref{thm:sphere,thm:ball} from the following theorem, which is a generalisation of a result of Fiorini, Rothvo{\ss} and Tiwary~\cite{FRT12} in dimension $d=2$. A very similar (actually slightly stronger) result appeared as \cite[Theorem~3.3]{Shi19b}.

\begin{thm}\label{thm:lower}
Fix $d\in\NN$ and let $P\su \RR^d$ be a $d$-dimensional polytope. Let $L$
be the field extension of $\QQ$ generated by the coordinates of the
vertices of $P$, and let $g$ be the transcendence degree of $L$ over $\QQ$.
Then $\xc(P)\ge \sqrt g$.
\end{thm}

We remark that Padrol (see \cite[Theorem 2(2)]{Pad16}) also proved a similar result with a slightly stronger bound that holds almost surely for polytopes $P$ drawn from continuous probability distributions. Padrol's result also implies the lower bound in \cref{thm:ball}.

\begin{proof}[Proof of \cref{thm:lower}]
Let $k=\xc(P)$ be the extension complexity of the polytope $P$. Then $P$ can be obtained as the image of some $d'$-dimensional polytope $Q\su \RR^{d'}$ with $k$ facets under a projection onto a $d$-dimensional subspace of $\RR^{d'}$. After a linear transformation of $\RR^{d'}$, we may assume that this projection is the projection $\RR^{d'}\to \RR^d$ onto the first $d$ coordinates.

Now, the polytope $Q\su \RR^{d'}$ is defined by a system of inequalities $Ax\le b$, where $A\in \RR^{k\times d'}$ and $b\in \RR^{k}$. Since every $d'$-dimensional polytope has at least $d'+1$ facets, we have $k\ge d'+1$. Let $N=k(d'+1)\le k^2$, and denote the entries of $A$ and $b$ by $\beta_1,\dots,\beta_N$ (in any order).

Note that the coordinates of the vertices of $Q$ can be expressed as rational functions of $\beta_1,\dots, \beta_N$, because each vertex is the unique solution of a linear system of equations whose coefficients are among $\beta_1,\dots, \beta_N$ (indeed, these linear equations are given by the equality cases of the constraints $Ax\le b$ which the vertex satisfies). This means that the coordinates of the vertices of $Q$ all lie in the field $\QQ(\beta_1,\dots,\beta_N)$.

Each vertex of $P$ can be obtained as the projection of some vertex of $Q$ (where the projection is onto the first $d$ coordinates). Hence the coordinates of the vertices of $P$ also all lie in the field $\QQ(\beta_1,\dots,\beta_N)$, so $L\su \QQ(\beta_1,\dots,\beta_N)$. It follows that $g=\operatorname{trdeg}(L/\QQ)\le \operatorname{trdeg}(\QQ(\beta_1,\dots,\beta_N)/\QQ) \le N\le k^2$, and $\xc(P)=k\ge \sqrt g$.
\end{proof}

It is not hard to deduce the lower bounds on the extension complexity in \cref{thm:sphere} and \cref{thm:ball} from \cref{thm:lower}. Indeed, note that in the setting of \cref{thm:ball}, with probability $1$ the coordinates of all the vertices of $P$ are algebraically independent. Furthermore, the number of vertices of $P$ is a.a.s.\ of the form $\Theta(m^{(d-1)/(d+1)})=\Theta(n)$ (see for example \cite{Rei05}). Thus, when applying \cref{thm:lower} to $P$ we a.a.s.\ have $g=d\cdot \Theta(n)=\Theta(n)$ and obtain $\xc(P)\geq \Theta(\sqrt{n})$.

In the setting of \cref{thm:sphere}, with probability $1$ we have $g=(d-1)n$ when applying \cref{thm:lower} to the $n$-vertex polytope $P$ and obtain $\xc(P)\geq \sqrt{(d-1)n}$ (note that the $d$ coordinates of each vertex of $P$ satisfy the algebraic relation $x_1^2+\dots+x_d^2=1$, but taking $d-1$ of the $d$ coordinates for each of the $n$ vertices gives with probability $1$ a transcendence basis of size $(d-1)n$).

\section{Basic lemmas for the random polytope upper bound}
\label{sec:basic-lemmas}

In this section we collect a number of basic facts about spherical caps and the vertex and facet distribution of random polytopes. These lemmas will be used later in the proofs of the upper bounds in \cref{thm:sphere,thm:ball}.

\subsection{Spherical caps}
As before, let $B\su \RR^d$ be the closed unit ball centred at the origin, and let $S\su \RR^d$ be the $(d-1)$-dimensional unit sphere (the boundary of $B$). For any two points $x,y\in S$ the spherical distance between $x$ and $y$ is the length of the shortest arc on the sphere $S$ connecting $x$ and $y$. Note that this length is equal to the angle between the points $x$ and $y$, measured from
the origin. In particular, the spherical distance between a pair of points $x,y\in S$ is at least $0$ and at most $\pi$.

Given a point $p\in S$, and $0<\eps<\pi$, the \emph{spherical cap} with radius $\eps$ centred at $p$ is the subset $X\su S$ of all points on $S$ with spherical distance at most $\eps$ from $p$. We call the convex hull $\conv(X)\su B$ of this subset $X\su S$ the \emph{solid cap} with radius $\eps$ centred at $p$. The point $p$ is called the \emph{centre} of this spherical cap and of the corresponding solid cap. Note that by definition the centre of any spherical or solid cap is always a point on the unit sphere $S$.

\begin{fact}
\label{fact:volume}
Fix $d\geq 2$. Then the surface area of a spherical cap of radius $\eps$ is $\Theta(\eps^{d-1})$, and the volume of a solid cap of radius $\eps$ is $\Theta(\eps^{d+1})$.
\end{fact}

\cref{fact:volume} can be deduced from exact formulas for the surface areas and volumes of caps (see for example \cite{Li11}), or can be computed directly by some elementary geometric estimates.

Note that any hyperplane $H\su \RR^d$ intersecting the interior of the unit ball $B$ cuts the ball $B$ into two solid caps, one on either side of $H$ (where we consider the intersection $H\cap B$ to be part of both of these solid caps). The centres of these two solid caps are the two intersection points of the sphere $S$ with the line orthogonal to $H$ through the origin (the centre of the ball $B$).

We will need the following basic packing and covering lemma for caps on the sphere.

\begin{lem}
\label{lem:covering-packing}Fix $d\geq 2$, and consider the $d$-dimensional unit ball $B\su \RR^d$ and the $(d-1)$-dimensional
unit sphere $S\subseteq\RR^{d}$. Then for any $0<\eps<\pi/50$ we can find a collection $A\subseteq S$ of $O(\eps^{1-d})$ points on the sphere, with the following properties.
\begin{enumerate}
    \item[(I)] Every pair of points in $A$ is separated by a spherical distance of at least $\eps/2$.
    \item[(II)] For any $p\in S$, there are $O(1)$ points in $A$ within spherical distance $30\varepsilon$ of $p$.
    \item[(III)] Every solid cap of radius $\eps/2$ in the ball $B$ is fully contained in a solid cap of radius $\eps$ centred at some point $a\in A$.
\end{enumerate}
\end{lem}

\begin{proof}
Let $A$ be any maximal collection of points on the sphere $S$ satisfying (I). We claim that all the other properties are automatically satisfied.

First, note that all spherical caps of radius $\eps/8$ centred at points in $A$ are disjoint from each other. Since each of these spherical caps has surface area $\Theta(\eps^{d-1})$, we obtain that $\vert A\vert\leq O(\eps^{1-d})$.

For (II), consider any point $p\in S$ and let $Z\subseteq A$ be the set of all points in $A$ within spherical distance $30\eps$ of $p$. Then the spherical caps of radius $\eps/8$ centred at all points $z\in Z$ are all disjoint and have a total surface area of $\Theta(|Z| \eps^{d-1})$. Also, all these caps are contained in the spherical cap of radius $31\eps$ centred at $p$, which has surface area $\Theta(\eps^{d-1})$. It follows that $|Z|=O(1)$.

For (III), consider any solid cap of radius $\eps/2$ centred at some point $p\in S$. Let $C\su S$ be the corresponding spherical cap. By maximality of $A$, there is some $a\in A$ within spherical distance $\eps/2$ of $p$, so $C$ is fully contained in the spherical cap of radius $\eps$ centred at $a$. Hence the original solid cap of radius $\eps/2$ centred at $p$ is contained in the solid cap of radius $\eps$ centred at $a$.
\end{proof}

\begin{lem}\label{lem:colouring-point-collection}
Fix $d\geq 2$. Let $0<\eps<\pi/50$ and let $A\subseteq S$ be a collection of points on the $(d-1)$-dimensional unit sphere $S$ with the properties in \cref{lem:covering-packing}. Then we can colour the points in $A$ with $O(1)$ colours such that any two points of the same colour have spherical distance at least $30\eps$.
\end{lem}
\begin{proof}
Let us consider an auxiliary graph $G$ with vertex set $A$, where there is an edge between two elements of $A$ if they are within spherical distance
$30\eps$ from each other. By property (II) of \cref{lem:covering-packing}, the graph $G$ has maximum degree $O(1)$. We can therefore greedily colour the points in $A$ as desired.
\end{proof}

\subsection{Random polytopes}

Next, we prove some lemmas about the way facets and vertices are typically distributed in random polytopes, in the settings of \cref{thm:sphere,thm:ball}. First, we need to know the typical number of vertices in the setting of \cref{thm:ball}. The expected number of vertices is a classical result (in two dimensions this was computed
by R\'enyi and Sulanke~\cite{RS63} in their foundational paper
on random polygons, and in higher dimensions it seems to have been
first computed by Raynaud~\cite{Ray70}). Concentration results are more recent; the following theorem is a special case of, for example, \cite[Theorem~2.11]{Vu05} or \cite[Theorem~5]{Rei05}.

\begin{thm}
\label{thm:number-vertices-ball}Fix $d\geq 2$ and let $P$ be the convex
hull of $m$ random points in the 
unit ball $B\subseteq\RR^{d}$. Then, a.a.s.\ $P$ has $\Theta(m^{(d-1)/(d+1)})$ vertices.
\end{thm}

Next, we need the fact that in the settings of both \cref{thm:sphere,thm:ball}, all facets are ``inside'' small caps, for a suitable notion of ``inside'' which we now define.

\begin{defn}
For a polytope $P\su B$, we say that a facet $f$ of $P$ is \emph{encapsulated} by a solid cap $C$ if, writing $H_f$ for the hyperplane containing $f$, we have $H_f\cap B\su C$.
\end{defn}

\begin{lem}
\label{lem:cap-diameter-sphere}Fix $d\geq 2$ and let $P$ be the convex
hull of $n$ random points on the $(d-1)$-dimensional
unit sphere $S\subseteq\RR^{d}$. Then, with probability $1-n^{-\omega(1)}$, each facet of $P$ is encapsulated by some solid cap of radius $n^{-1/(d-1)}\log^2 n$.
\end{lem}

\begin{lem}
\label{lem:cap-diameter-ball}Fix $d\geq 2$, let $P$ be the convex
hull of $m$ random points in the 
unit ball $B\subseteq\RR^{d}$, and let $n=m^{(d-1)/(d+1)}$. Then, with probability $1-n^{-\omega(1)}$, each facet of $P$ is encapsulated by some solid cap of radius $n^{-1/(d-1)}\log^2 n$.
\end{lem}

The statements of \cref{lem:cap-diameter-sphere,lem:cap-diameter-ball} can be interpreted as saying that random polytopes in the unit ball are quite close to ``filling out'' the whole ball. There are a large number of related results in the literature: for example, \cref{lem:cap-diameter-ball} can be deduced from a result by B\'ar\'any and Dalla~\cite{BD97}. However, we believe it is simplest to provide a simple self-contained and unified proof of \cref{lem:cap-diameter-sphere,lem:cap-diameter-ball}.

\begin{proof}[Proof of \cref{lem:cap-diameter-sphere,lem:cap-diameter-ball}]
Let us enumerate the random points as $p_1, p_2,\dots, p_m$ in the order they were chosen (where $m=n$ in the setting of \cref{lem:cap-diameter-sphere}). Note that with probability one, no $d+1$ of the points $p_1, p_2,\dots, p_m$ lie on a common hyperplane, meaning our polytope $P=\conv(p_1,\dots,p_m)$ is simplicial (all of its facets have exactly $d$ vertices). For each subset $I\su [m]$ of size $\vert I\vert =d$, let $E_I$ be the event that $\conv(p_i\,|\,i\in I)$ is a facet of $P$ which is not encapsulated by a solid cap of radius $n^{-1/(d-1)}\log^2 n$. It suffices to show that $\Pr(E_I)\le n^{-\omega(1)}$ for each $I$; we may then take the union bound over all $\binom{m}{d}\leq m^d\leq n^{d(d+1)/(d-1)}$ different sets $I$.

We may assume without loss of generality that $I=\lbrace 1,\dots,d\rbrace$. Let us condition on any outcome of the random points $p_1,\dots,p_d$; what we will actually show is the stronger fact that $\Pr(E_I\,|\,p_1,\dots,p_d)\le n^{-\omega(1)}$. The hyperplane $H$ through the points $p_1,\dots,p_d$ cuts the ball $B$ into two solid caps $C$ and $C'$. If one of these two solid caps has radius at most $n^{-1/(d-1)}\log^2 n$, then trivially $\Pr(E_I\,|\,p_1,\dots,p_d)=0$ (because if $p_1,\dots,p_d$ form a facet of $P$, then this facet is encapsulated by both of the solid caps $C$ and $C'$).

So let us now consider the case that both of the solid caps $C$ and $C'$ have radius at least $n^{-1/(d-1)}\log^2 n$. If $p_1,\dots,p_d$ form a facet of $P$, then all the remaining points $p_{d+1},\dots,p_m$ must lie on the same side of the hyperplane $H$ through $p_1,\dots,p_d$. This means that one of the solid caps $C$ or $C'$ must contain all of the points $p_{d+1},\dots,p_m$, while the other one contains none of them.

Each of the points $p_{d+1},\dots,p_m$ lies in the solid cap $C$ with probability at least $\Omega(m^{-1}\log^2 n)$. Indeed, in the setting of \cref{lem:cap-diameter-sphere} the surface area of $C$ is at least $\Omega(n^{-1}\log^{2(d-1)} n)=\Omega(m^{-1}\log^{2(d-1)} n)$. In the setting of \cref{lem:cap-diameter-ball} the volume of $C$ is at least $\Omega(n^{-(d+1)/(d-1)}\log^{2(d+1)} n)=\Omega(m^{-1}\log^{2(d+1)} n)$. Thus, the probability that none of the points $p_{d+1},\dots,p_m$ lies in $C$ is of the form $\left(1-\Omega(m^{-1}\log^2 n)\right)^{m-d}\leq \exp(-\Omega(\log^2 n))\le n^{-\omega(1)}$. Similarly, the probability that the cap $C'$ contains none of the points $p_{d+1},\dots,p_m$ is $n^{-\omega(1)}$. This proves that $\Pr(E_I\,|\,p_1,\dots,p_d)\le n^{-\omega(1)}$, as desired.
\end{proof}

Next, we will prove that the vertices of random polytopes are quite well-distributed, not being too ``clustered'' in any small cap. Again, this is true in the settings of both \cref{thm:sphere,thm:ball}.

\begin{lem}
\label{lem:sphere-distributed}Fix $d\geq 2$, let $n\in \NN$, and let $C$ be a solid cap with radius $\varepsilon\geq n^{-1/(d-1)}\log^2 n$ in the unit ball $B\su \RR^d$. Let $P$ be the convex hull of $n$ random points on the $(d-1)$-dimensional unit sphere $S\subseteq\RR^{d}$. Then, with probability $1-n^{-\omega(1)}$, the solid cap $C$ contains at most $O(\varepsilon^{d-1} n)$ vertices of the polytope $P$.
\end{lem}

\begin{lem}
\label{lem:ball-distributed} Fix $d\geq 2$, let $m\in \NN$, and define $n=m^{(d-1)/(d+1)}$. Let $C$ be a solid cap with radius $\eps\ge n^{-1/(d-1)}\log^{60} n$ in the unit ball $B\su \RR^d$. Let $P$ be the convex hull of $m$ random points in the ball $B$. Then, with probability $1-n^{-\omega(1)}$, the solid cap $C$ contains at most $O(\varepsilon^{d-1} n)$ vertices of the polytope $P$.
\end{lem}

The proof of \cref{lem:sphere-distributed} is extremely simple; one basically just applies a Chernoff bound.

\begin{proof}[Proof of \cref{lem:sphere-distributed}]
The vertices of $P$ are precisely the $n$ random points chosen to define $P$. Note that by \cref{fact:volume} the surface area of $C$ is $\Theta(\eps^{d-1})$, so each of the $n$ random points lies in $C$ with probability $\Theta(\eps^{d-1})$. Thus, the expected number of points in $C$ is $\Theta(\eps^{d-1})n\geq \log^{2(d-1)}n$. So by the Chernoff bound, with probability $1-n^{-\omega(1)}$ the number of points in $C$ is at most twice its expectation, and therefore of the form $O(\varepsilon^{d-1} n)$.
\end{proof}

The proof of \cref{lem:ball-distributed} is more involved (this is actually the only significant difference between the proofs of \cref{thm:ball} and \cref{thm:sphere}). \cref{lem:ball-distributed} will be a consequence of the following bound for the volume of the complement of a random polytope intersected with a fixed solid cap.

\begin{lem}\label{lem:volume-concentration}
Fix $d\geq 2$, let $m\in \NN$, and define $n=m^{(d-1)/(d+1)}$. Let $C$ be a solid cap with radius $\eps\ge n^{-1/(d-1)}\log^{20} n$ in the unit ball $B\su \RR^d$. Let $P$ be the convex hull of $m$ random points in the ball $B$. Then with probability $1-n^{-\omega(1)}$ we have $\Vol(C\setminus P)\leq O(\varepsilon^{d-1}n^{-2/(d-1)})=O(\varepsilon^{d-1}m^{-2/(d+1)})$.
\end{lem}

Before proving \cref{lem:volume-concentration}, we show how it implies \cref{lem:ball-distributed}.

\begin{proof}[Proof of \cref{lem:ball-distributed}]
Let $X=(p_{1},\dots,p_{m})$ be the sequence of random points defining $P$, and let $T=\lceil n/\eps^2\rceil$. Note that then every integer $t=T,\dots,m$ satisfies
\begin{equation}\label{eq:condition-epsilon-check}
t^{-1/(d+1)}\log^{20}(t^{(d-1)/(d+1)})\leq T^{-1/(d+1)}\log^{20} n\leq n^{-1/(d+1)}\eps^{2/(d+1)}\log^{60(d-1)/(d+1)} n\leq \eps.
\end{equation}

Recall from \cref{fact:volume} that the solid cap $C$ has volume $\Theta(\eps^{d+1})$, so the expected number of indices $i\in \{1,\dots,T\}$ with $p_i\in C$ is $\Theta(\eps^{d+1}T)=\Theta(\eps^{d-1}n)\geq \log^{60(d-1)} n$. Hence a Chernoff bound implies that with probability $1-n^{-\omega(1)}$ there are at most $O(\eps^{d-1}n)$ indices $i\in \{1,\dots,T\}$ with $p_i\in C$ (and hence in particular at most $O(\eps^{d-1}n)$ vertices $p_i\in C$ with $1\leq i\leq T$).

Now for every integer $k$ with $0\leq k < \log_2(m/T)$, let us bound the number of vertices $p_i\in C$ of $P$ with $2^{k}T < i \leq 2^{k+1}T$. Note that this number of vertices is at most the number $Z_k$ of indices $i$ with $2^{k}T < i \leq 2^{k+1}T$ and $p_i\in C\setminus \conv(p_1,\dots,p_{2^kT})$. We claim that for each $k$, we have $Z_k\leq O(\eps^{d-1}(2^kT)^{(d-1)/(d+1)})$ with probability $1-n^{-\omega(1)}$. Indeed, by \cref{lem:volume-concentration} (using \cref{eq:condition-epsilon-check}) with probability $1-n^{-\omega(1)}$ we have $\Vol(C\setminus \conv(p_1,\dots,p_{2^kT}))=O(\eps^{d-1}(2^kT)^{-2/(d+1)})$. Conditioning on any such outcome of $p_1,\dots,p_{2^kT}$, a Chernoff bound shows that $Z_k\leq O(\eps^{d-1}(2^kT)^{(d-1)/(d+1)})$ with probability $1-n^{-\omega(1)}$. Overall this indeed shows that with probability $1-n^{-\omega(1)}$ we have $Z_k\leq O(\eps^{d-1}(2^kT)^{(d-1)/(d+1)})$ and so there are at most $O(\eps^{d-1}(2^kT)^{(d-1)/(d+1)})$ vertices $p_i\in C$ of $P$ with $2^{k}T < i \leq 2^{k+1}T$.

Hence with probability at least $1-(2+\log_2 m)\cdot n^{-\omega(1)}=1-n^{-\omega(1)}$, the total number of vertices of $P$ in $C$ is at most
\begin{align*}
&O(\eps^{d-1}n)+\sum_{k=0}^{\lfloor \log_2(m/T)\rfloor} O\left(\eps^{d-1}(2^kT)^{(d-1)/(d+1)}\right)\\ 
&\qquad= O(\eps^{d-1}n) + O(\eps^{d-1})\cdot T^{(d-1)/(d+1)} \sum_{k=0}^{\lfloor \log_2(m/T)\rfloor} 2^{k(d-1)/(d+1)}\\
&\qquad= O(\eps^{d-1}n) + O(\eps^{d-1})\cdot T^{(d-1)/(d+1)}\cdot O((m/T)^{(d-1)/(d+1)})\\
&\qquad= O(\eps^{d-1}n) + O(\eps^{d-1}m^{(d-1)/(d+1)}) = O(\eps^{d-1}n),
\end{align*}
as desired.
\end{proof}

To prove \cref{lem:volume-concentration}, we will need some auxiliary results from the literature. First,
we need the approximate expected volume of a random polytope (this estimate is classical, having been 
first computed by Raynaud~\cite{Ray70}).
\begin{thm}
\label{lem:expected-vertices}Let $P$ be a random polytope as in \cref{lem:ball-distributed}. Then the expected volume of $B\setminus P$ is $\Theta(m^{-2/(d+1)})=\Theta(n^{-2/(d-1)})$.
\end{thm}

Second, we will need the following concentration inequality by Boucheron, Lugosi and Massart~\cite[Corollary~3]{BLM03}. This is, in some sense, a version of the well-known Efron--Stein inequality with an exponential tail bound.
\begin{lem}
\label{lem:BLM}Let $X=(X_{1},\dots,X_{m})$ be a sequence
of independent random variables, and let $X^{(i)}$ be
obtained from $X$ by replacing $X_{i}$ with an independent copy.
Let $f(X)$ be any function of $X$, and suppose $\nu> 0$ is such that the outcome of $X$ satisfies
\[
\E\left[\sum_{i=1}^{m} (f(X)-f(X^{(i)}))^{2}\middle|X\right]\le \nu
\]
with probability 1. Then $\Pr(f(X)>\E f(X)+t)\le e^{-t^2/(4\nu)}$ for all $t>0$.
\end{lem}

Now we prove \cref{lem:volume-concentration}.

\begin{proof}[Proof of \cref{lem:volume-concentration}]

Let $F$ be the set of points in the ball $B$ which are not contained in any solid cap of radius $n^{-1/(d-1)}\log^2 n$. Equivalently, $F$ is the open ball with radius $\cos (n^{-1/(d-1)}\log^2 n)=1-\Theta(n^{-2/(d-1)}\log^4 n)$ around the origin $0$. Note that $\Vol(B\setminus F)=\Theta(n^{-2/(d-1)}\log^4 n)$. Furthermore, using the first part of \cref{fact:volume}, we have
\[\Vol(C\setminus F)\leq \Vol(\conv(C\cup \{0\})\setminus F) =\Theta(\eps^{d-1})\Vol(B\setminus F)=\Theta(\eps^{d-1}n^{-2/(d-1)}\log^4 n).\]

Let $X=(p_{1},\dots,p_{m})$ be the sequence of random points defining $P$. We claim that with probability $1 - n^{-\omega(1)}$ we have $F\su P$. Indeed, by \cref{lem:cap-diameter-ball}, with probability $1 - n^{-\omega(1)}$ all facets of $P$ are disjoint from $F$, meaning that either $F\su P$ or $F\cap P=\emptyset$. However, note that with probability at least $1 - n^{-\omega(1)}$ we have $p_j\in F$ for some $j\in \{1,\dots,m\}$, and so in particular $F\cap P\neq \emptyset$. This shows that we indeed have $F\su P$ with probability $1 - n^{-\omega(1)}$.

Let us now define $Q=\conv(F \cup \{p_1,\dots,p_m\})$ and $f(X)=\Vol(C\setminus Q)$. Then with probability $1 - n^{-\omega(1)}$ we have $Q=P$, so in order to prove the lemma it suffices to show that with probability $1- n^{-\omega(1)}$ we have $f(X)\leq O(\eps^{d-1}n^{-2/(d-1)})$.

Note that always $f(X)\leq \Vol(C \setminus P)\leq \Vol(\conv(C \cup \{0\}) \setminus P)$. The spherical sector $\conv(C \cup \{0\})$ comprises a $O(\eps^{d-1})$-fraction of the ball $B$. So by symmetry and linearity of expectation, \cref{lem:expected-vertices} implies $\E [\Vol(\conv(C \cup \{0\}) \setminus P )]= \Theta(\eps^{d-1})\cdot \E[\Vol(B\setminus P)]=\Theta(\eps^{d-1}n^{-2/(d-1)})$ and therefore $\E [f(X)]\leq O(\eps^{d-1}n^{-2/(d-1)})$.

Now, as in \cref{lem:BLM}, let $X^{(i)}$ be obtained from $X$ by replacing $p_{i}$ with an independent random point $p_i'$ in the ball $B$, and let $Q^{(i)}=\conv(F \cup \{p_1,..,p_{i-1},p_i',p_{i+1},..,p_m\})$.
Then it suffices to prove that
\begin{equation}\label{eq:efron-stein-condition}
\E\left[\sum_{i=1}^{m}(f(X)-f(X^{(i)}))^{2}\middle|X\right]\le O\left(\eps^{2(d-1)} n^{-4/(d-1)}\log^{-2} n\right)
\end{equation}
for every outcome of $X$. Indeed, given \cref{eq:efron-stein-condition}, we can apply \cref{lem:BLM} with $t=\eps^{d-1}n^{-2/(d-1)}$, and conclude that with probability at least $1-e^{-\Omega(\log^{2} n)}=1-n^{-\omega(1)}$ we have $f(X)\leq \E[f(X)]+\eps^{d-1}n^{-2/(d-1)}=O(\eps^{d-1}n^{-2/(d-1)})$, as desired.

So let us fix an outcome of $X=(p_1,\dots,p_m)$ for the rest of this proof. For $i=1,\dots,m$, let $\hat{X}_i=(p_1,..,p_{i-1},p_{i+1},..,p_m)$ be the sequence of points obtained from $X$ by deleting $p_i$, and define $\hat Q_i=\conv(F \cup \{p_1,..,p_{i-1},p_{i+1},..,p_m\})$. By slight abuse of notation, let us write $f(\hat{X}_i)=\Vol(C\setminus Q^{(i)})$. We claim that the following statements hold for any outcome of the random points $p_1',\dots,p_m'$ and any $i=1,\dots,m$.
\begin{enumerate}
    \item[(A)] $0\leq f(\hat{X}_i)-f(X^{(i)})\leq \Vol(Q^{(i)} \setminus \hat{Q}_i)\leq O(n^{-(d+1)/(d-1)}\log^{2(d+1)} n)$.
    \item[(B)] $0\leq f(\hat{X}_i)-f(X)\leq \Vol(Q \setminus \hat{Q}_i)\leq O(n^{-(d+1)/(d-1)}\log^{2(d+1)} n)$.
    \item[(C)] If $f(\hat{X}_i) \neq f(X^{(i)})$, then $p_i'$ must be contained in a solid cap of radius $n^{-1/(d-1)}\log^2 n$ which intersects $C$.
    \item[(D)] Every point $p\in C\setminus F$ appears in $Q \setminus \hat{Q}_i$ for at most $d+1$ different indices $i$.
\end{enumerate}

In (A), the first two inequalities follow directly from the definitions of $f(X^{(i)})$ and $f(\hat{X}_i)$. To prove the last inequality, note that every point in $Q^{(i)} \setminus \hat{Q}_i$ must lie in some solid cap of radius $n^{-1/(d-1)}\log^2 n$ that contains $p_i'$. Indeed, for every $p\in Q^{(i)} \setminus \hat{Q}_i$, we can find a hyperplane $H_p$ such that $p$ is on one side of $H_p$ whereas $F$ (and also $p_1,\dots,p_{i-1},p_{i+1},\dots,p_m$) are on the other side. Then $H$ cuts  a solid cap of radius at most $n^{-1/(d-1)}\log^2 n$ out of the ball $B$, which contains $p$ and also $p_i'$ (since $p\in Q^{(i)}=\conv(F \cup \{p_1,..,p_{i-1},p_{i+1},..,p_m\})$). Hence $Q^{(i)} \setminus \hat{Q}_i$ is a subset of the union of all solid caps of radius $n^{-1/(d-1)}\log^2 n$ containing $p_i'$. The total volume of this union is at most the volume of a solid cap of radius $2n^{-1/(d-1)}\log^2 n$, and this volume is $O(n^{-(d+1)/(d-1)}\log^{2(d+1)} n)$ by \cref{fact:volume}.

The proof of (B) is analogous by considering the point $p_i$ instead of $p_i'$.

Note that the assumption in (C) implies that $(Q^{(i)} \setminus \hat{Q}_i)\cap C\neq \emptyset$. Consider some point $p\in (Q^{(i)} \setminus \hat{Q}_i)\cap C$ and recall from the argument for (A) that $p$ must lie in some solid cap of radius $n^{-1/(d-1)}\log^2 n$ containing $p_i'$. Since $p\in C$, this solid cap intersects $C$.

Finally, for (D), note that for any $p\in Q$, by Carath\'{e}odory's theorem there are indices $i_1,\dots,i_{d+1}$ such that $p\in \conv(F\cup\{p_{i_1},\dots,p_{i_{d+1}}\})$. In particular, for all $i\notin\{i_1,\dots,i_{d+1}\}$ we have $p\in \hat{Q}_i$ and therefore $p\notin Q \setminus \hat{Q}_i$.

Now, in order to show \cref{eq:efron-stein-condition}, let us first observe that we always have
\[\sum_{i=1}^{m}(f(X)-f(X^{(i)}))^{2}\leq 2\sum_{i=1}^{m}(f(X)-f(\hat{X}_i))^{2}+2\sum_{i=1}^{m}(f(\hat{X}_i)-f(X^{(i)}))^{2}\]
and hence
\begin{equation}\label{eq:efron-stein-condition-split}
    \E\left[\sum_{i=1}^{m}(f(X)-f(X^{(i)}))^{2}\middle|X\right]\leq 2\sum_{i=1}^{m}(f(\hat{X}_i)-f(X))^{2}+2\sum_{i=1}^{m} \E\left[(f(\hat{X}_i)-f(X^{(i)}))^{2}\middle| X\right].
\end{equation}

By (C), for every $i=1,\dots,m$, we can only have $f(\hat{X}_i)-f(X^{(i)})\neq 0$ if $p_i'$ is contained in the union of all solid caps of radius $n^{-1/(d-1)}\log^2 n\leq \eps$ which intersect $C$. Note that this union is a subset of the spherical shell $B\setminus F$, and by the first part of \cref{fact:volume} the volume of this union is at most $O((3\eps)^{d-1})\Vol(B\setminus F)=O(\eps^{d-1}n^{-2/(d-1)}\log^4 n)$. Hence the probability of having $f(\hat{X}_i)-f(X^{(i)})\neq 0$ is at most $O(\eps^{d-1}n^{-2/(d-1)}\log^4 n)$. Given (A), this implies
\begin{align*}
\sum_{i=1}^{m}\E\left[(f(\hat{X}_i)-f(X^{(i)}))^{2}\middle| X\right]
&\le \sum_{i=1}^{m} O(\eps^{d-1}n^{-2/(d-1)}\log^4 n)\cdot O(n^{-2(d+1)/(d-1)}\log^{4(d+1)} n)\\
&= n^{(d+1)/(d-1)}\cdot O(\varepsilon^{d-1}n^{-2(d+2)/(d-1)} \log^{4(d+2)} n)\\
&= O(\eps^{d-1} n^{-(d+3)/(d-1)} \log^{4(d+2)} n)\leq O(\eps^{2(d-1)} n^{-4/(d-1)}\log^{-2} n),\end{align*}
where we used that $m=n^{(d+1)/(d-1)}$ and $\eps^{d-1}\ge n^{-1}\log^{20(d-1)} n\geq n^{-1}\log^{4(d+3)} n$.

It remains to bound the first sum on the right-hand side of \cref{eq:efron-stein-condition-split}. Note that we have $f(\hat{X}_i)-f(X)=\Vol(C\setminus \hat{Q}_i)-\Vol(C\setminus Q)=\Vol((Q \setminus \hat{Q}_i)\cap C)$ and $(Q \setminus \hat{Q}_i)\cap C\su C\setminus F$ for every $i$. By (D), every point in $C\setminus F$ can appear in $(Q \setminus \hat{Q}_i)\cap C$ for at most $d+1$ indices $i$, so we obtain
\[\sum_{i=1}^{m} (f(\hat{X}_i) - f(X)) = \sum_{i=1}^{m} \Vol((Q \setminus \hat{Q}_i)\cap C)\leq (d+1) \cdot \Vol(C\setminus F) \leq O(\eps^{d-1} n^{-2/(d-1)} \log^4 n).\]
Using (B), this implies
\begin{align*}
\sum_{i=1}^m (f(\hat{X}_i)-f(X))^2 &\leq O(n^{-(d+1)/(d-1)}\log^{2(d+1)} n) \sum_{i=1}^{m} (f(\hat{X}_i) - f(X))\\
&\le O(\eps^{d-1} n^{-(d+3)/(d-1)} \log^{2(d+3)} n)\le O(\eps^{2(d-1)} n^{-4/(d-1)}\log^{-2} n),
\end{align*}
where we again used that $\eps^{d-1}\geq n^{-1}\log^{4(d+3)} n$.

All in all, we can conclude that the left-hand side of \cref{eq:efron-stein-condition-split} is bounded by $O(\eps^{2(d-1)} n^{-4/(d-1)}\log^{-2} n)$, showing \cref{eq:efron-stein-condition}.
\end{proof}

\section{The ``lampshade'' argument}
\label{sec:lampshade}

The following lemma drives the proof of \cref{thm:sphere,thm:ball}. It is inspired by (and very closely related to) a lemma in Shitov's paper proving that polygons have sublinear extension complexity~\cite[Lemma 3.1]{Sh14v1}. Informally, \cref{lem:shitov} states the following. For a polytope $P\su B$, let us consider a set of facets $F'$ and a set of vertices $V'$, such that the facets in $F'$ are far away from the vertices in $V'$. Then the submatrix of the slack matrix of $P$ containing only the slacks between the facets in $F'$ and the vertices in $V'$ has bounded nonnegative rank. This even remains true if we are allowed to modify this submatrix by subtracting rows corresponding to the slacks of vertices of $P$ that are close to the facets in $F'$.

\begin{lem}
\label{lem:shitov}
For any $d\geq 2$ there is $R\in \NN$ such that the following holds. Let $B\subseteq\RR^{d}$ be the unit ball in $\RR^{d}$ and $S\subseteq B$ be the boundary of $B$ (i.e.\ the unit sphere). Let $P\su B$ be a polytope, and let $V$ and $F$ be the set of vertices and the set of facets of $P$, respectively. Furthermore, let $a\in S$ be a point, let $0<\eps<\pi/5$, and let $X\subseteq B$ be the solid cap with radius $\eps$ centred at $a$. Also, let $Y\subseteq B$ be the convex hull of all points on $S$ with spherical distance at least $5\varepsilon$ from $a$ (this is the solid cap of radius $\pi-5\varepsilon$ centred at the point antipodal from $a$).

Now, suppose that $F'\su F$ is a subset of the facets of $P$ such that every facet in $F'$ is encapsulated by $X$. Furthermore, suppose that $V'\su V$ is a subset of the vertices of $P$ such that $V'\su Y$.

Let $M$ be a slack matrix of the polytope $P$, with rows indexed by $V$ and columns indexed by $F$. Let $M'$ be a matrix with rows indexed by $V'$ and columns indexed by $F'$, such that, for each $v\in V'$, at least one of the following conditions holds:
\begin{itemize}
    \item[(1)] $M_{v,f}'=M_{v,f}$ for all $f\in F'$, or
    \item[(2)] there is $x_v\in V\cap X$ such that $M_{v,f}'=M_{v,f}-M_{x_v,f}$ for all $f\in F'$.
\end{itemize}
Then the matrix $M'$ has nonnegative entries and $\rank_+ M'\le R$.
\end{lem}

The most important ingredient of the proof of \cref{lem:shitov} is the following geometric lemma. The convex set $Q$ in this lemma corresponds to the ``lampshade'' described in the proof outline in \cref{sect-outline}.

\begin{lem}
\label{lem:cap-polyhedron}For every integer $d\geq 2$, there is a constant $R\in \NN$ such that the following holds. Let $B\subseteq\RR^{d}$ be the unit ball in $\RR^{d}$, and let $S\subseteq B$ be the boundary of $B$. Let $a\in S$ be a point, let $0<\eps<\pi/5$, and let $X\subseteq B$ be the solid cap with radius $\varepsilon$ centred at $a$. Furthermore, let $Y\subseteq B$ be the convex hull of all points on $S$ with spherical distance at least $5\varepsilon$ from $a$. Then there is a convex subset $Q\su \RR^d$ such that the following conditions hold:
\begin{itemize}
\item[(i)] Every affine hyperplane $H\su \RR^{d}$ intersecting the interior of $B$ such that $H\cap B\su X$ satisfies $H\cap Q=\emptyset$.
\item[(ii)] For any points $x\in X$ and $y\in Y$, we have $\left\lbrace y+t(y-x)\mid t\in \RR_{\geq 0}\right\rbrace\su Q$.
\item[(iii)] For any finite set of points $A\su Q$, we can find a set of points $A'\su Q$ of size $\vert A'\vert\leq R$ such that $A\su \conv(A')$.
\end{itemize}
\end{lem}

Condition (ii) has the following geometric meaning. For any points $x\in X$ and $y\in Y$, consider the ray along the line through $x$ and $y$ starting at $y$ and pointing away from $x$. Condition (ii) states that this ray is entirely contained in $Q$. Note that this implies in particular that $Y\subseteq Q$.

We defer the proof of \cref{lem:cap-polyhedron} to \cref{sect-cone}. As some rough intuition, note that if we take $Q$ to simply be the union of all the rays in condition (ii) (so $Q$ is a cone with its tip cut off) then condition (i) is satisfied (but condition (iii) fails). To prove \cref{lem:cap-polyhedron}, we approximate this naive choice of $Q$ with a polyhedron, satisfying property (iii) while preserving properties (i) and (ii).

In order to deduce \cref{lem:shitov} from \cref{lem:cap-polyhedron}, we need the following well-known fact about sets of matrices with bounded nonnegative rank (for a proof of this fact, see for example \cite[Theorem 3.1]{BCR11} or \cite[Proposition 6.2]{LC09}).

\begin{fact}\label{fact:closed}
For any $m,n,r\in \NN$, let $\mathcal M_r\subseteq \RR_{\geq 0}^{m\times n}$ be the set of all nonnegative $m\times n$ matrices with nonnegative rank at most $r$. Then $\mathcal M_r$ is a closed set.
\end{fact}

We are now ready to prove \cref{lem:shitov}.

\begin{proof}[Proof of \cref{lem:shitov}]
Let us choose the constant $R\in \NN$ as in \cref{lem:cap-polyhedron}. First note that the statement of \cref{lem:shitov} is trivially true if one of the sets $F'$ and $V'$ is empty (because then the matrix $M'$ is empty). We may therefore assume that there exists at least one facet $f\in F'$. Since $f$ is encapsulated by $X$, all vertices of $f$ lie in the solid cap $X$. This means that $V\cap X\neq \emptyset$. Similarly, since $V'\su V\cap Y$, we may assume that $V\cap Y\neq \emptyset$.

Let $\mathcal M^*$ be the set of matrices obtainable as follows. For each $v\in V'$, choose some $\alpha_v\in [0,1)$ and $x_v\in V\cap X$. Let $M^*$ be the $V'\times F'$ matrix with entries $M^*_{v,f}=M_{v,f}-\alpha_v M_{x_v,f}$ for all $f\in F'$. Note that the matrix $M'$ does not quite lie in the set $\mathcal M^*$, but it does lie in the \emph{closure} of $\mathcal M^*$ (we would like to take each $\alpha_v\in \{0,1\}$, but we are only allowing $\alpha_v\in[0,1)$). We will show that each $M^*\in \mathcal M^*$ has nonnegative entries and $\rank_+ M^*\le R$, so \cref{fact:closed} will imply that the same holds for $M'$, as desired.

The key is to interpret the entries of $M^*$ geometrically, as follows. For each $v\in V'\su Y$, define \[w_v=\frac{1}{1-\alpha_v}v-\frac{\alpha_v}{1-\alpha_v} x_v=v+\frac{\alpha_v}{1-\alpha_v} (v-x_v).\]
Every facet $f\in F'$ corresponds to some constraint $a_f\cdot x\leq b_f$ that is used for the slack matrix $M$; define the affine-linear function $\psi_f:x\in \RR^d\to \RR$ by $\psi_f(x)=b_f-a_f\cdot x$, measuring the slack of a point $x$ with respect to $f$. Note that for any $f\in F'$ and $v\in V'$ we have $\psi_f(w_v)=\frac{1}{1-\alpha_v}\psi_f(v)-\frac{\alpha_v}{1-\alpha_v}\psi_f(x_v)$, so
\begin{equation}\label{eq:M*-slack}
(M^*_{v,f})_{f\in F'}=(M_{v,f}-\alpha_v M_{x_v,f})_{f\in F'}=(\psi_f(v)-\alpha_v \psi_f(x_v))_{f\in F'}=(1-\alpha_v)(\psi_f(w_v))_{f\in F'}.
\end{equation}

Now, let $Q$ be the convex set guaranteed by \cref{lem:cap-polyhedron}. Condition (ii) implies that each $w_v\in Q$, so condition (iii) ensures the existence of a set $W\su Q$ of size $\vert W\vert\leq R$, with $w_v\in \conv(W)$ for each $v\in V'$. Also, condition (i) implies that for any point $w\in Q$, we have $\psi_f(w)\geq 0$ for all $f\in F'$. To see this, fix a facet $f\in F'$ and consider the hyperplane $H_f\su \RR^d$ through $f$. Since this hyperplane contains the facet $f$, it intersects the interior of the ball $B$. As $f\in F'$ is encapsulated by $X$, we have $H_f\cap B\su X$. Therefore, by condition (i) we have $H_f\cap Q=\emptyset$. Thus, the entirety of $Q$ lies on the same side of the hyperplane $H_f$. Since the set $Q$ contains all vertices in $V\cap Y\su P$ (and $V\cap Y\neq \emptyset$ and $V\cap Y$ is disjoint from $H_f$), $Q$ lies on the same side of the hyperplane $H_f$ as our polytope $P$ does. So we indeed have $\psi_f(w)\geq 0$ for all $f\in F'$ and all $w\in Q$.

For each $w\in W\subseteq Q$, define the vector $u_w=(\psi_f(w))_{f\in F'}$, which by the previous paragraph has nonnegative entries. For each $v\in V'$, since $w_v\in \conv(W)$, the vector $(\psi_f(w_v))_{f\in F'}$ can be written as a convex combination of our vectors $u_w$. So, by \cref{eq:M*-slack}, each row $(M^*_{v,f})_{f\in F'}$ of $M^*$ can be written as a nonnegative linear combination of these vectors $u_w$. It follows that the entries of $M^*$ are nonnegative and that $\rank_+ M^*\leq \vert W\vert\leq R$, as desired.\end{proof}

\section{Upper-bounding the extension complexity of random polytopes}
\label{sec:upper-bound}

In this section we give a unified proof of the upper bounds in \cref{thm:sphere,thm:ball}. Let $\varepsilon=n^{-1/(2(d-1))}$, and assume that $n$ is sufficiently large (or, in the setting of \cref{thm:ball}, that $m$ is sufficiently large) such that $\eps<\pi/50$.

Let $A\su S$ be a collection of $O(\eps^{1-d})=O(\sqrt{n})$ points on the sphere $S$ with the properties in \cref{lem:covering-packing}. In particular, by property (III), every solid cap of radius $\eps/2$ is contained in a solid cap of radius $\eps$ centred at some point in $A$. Furthermore, by \cref{lem:colouring-point-collection} we can colour the points in $A$ with $\chi=O(1)$ colours such that any two points of the same colour have spherical distance at least $30\eps$. For $c=1,\dots,\chi$, let $A_c\su A$ be the set of points in $A$ with colour $c$ (then $A_1\cup\dots\cup A_c$ is a partition of $A$).

Recall that $P$ is a random polytope, given as the convex hull of $n$ random points on $S$ (in the setting of \cref{thm:sphere}) or as the convex hull of $m$ random points on $S$ (in the setting of \cref{thm:ball}). By \cref{lem:cap-diameter-sphere} or \cref{lem:cap-diameter-ball}, a.a.s.\ every facet of $P$ is encapsulated in some solid cap of radius $n^{-1/(d-1)}\log^2 n\le\varepsilon/2$, and is therefore also encapsulated in a solid cap of radius $\eps$ centred at some point in $A$. Furthermore, by \cref{lem:sphere-distributed} or \cref{lem:ball-distributed}, each solid cap of radius $5\eps$ centred at some point in $A$ contains with probability $1-n^{-\omega(1)}$ at most $O(\eps^{d-1}n)=O(\sqrt n)$ vertices of $P$. Since $\vert A\vert\leq O(\sqrt{n})$, we can conclude that a.a.s\ all solid caps of radius $5\eps$ centred at the points in $A$ contain each at most $O(\sqrt n)$ vertices of $P$. We will show that under these conditions we have $\xc(P)\leq O(\sqrt{n})$.

Let $V$ and $F$ be the sets of vertices and facets of $P$, and choose a partition $F=F_{1}\cup \dots \cup F_{\chi}$, such that for $c=1,\dots, \chi$ every facet in the set $F_{c}$ is encapsulated in a solid cap of radius $\eps$ centred at some point in $A_c$. For $c=1,\dots, \chi$, let $V_{c}\su V$ be the set of vertices of $P$ which
are contained in a solid cap of radius $5\varepsilon$ centred at some point in $A_c$, and let $W_c=V\setminus V_c$. Note that the vertices $w\in W_c$ are far away from the points $a\in A_c$: for each $w\in W_c$ and each $a\in A_c$, the vertex $w$ is contained in the convex hull of all points on $S$ with spherical distance at least $5\eps$ from $a$.

Now, consider a slack matrix $M$ of the polytope $P$ with rows indexed by $V$ and columns indexed by $F$. We partition $M$ into $2\chi$ submatrices $M[V_1,F_1],\dots,M[V_\chi,F_\chi]$ and $M[W_1,F_1],\dots,M[W_\chi,F_\chi]$, where for subsets $V'\su V$ and $F'\su F$, by $M[V',F']$ we denote the $V'\times F'$ submatrix of $M$ containing the slacks between vertices in $V'$ and facets in $F'$. For the rest of the proof, our goal will be to show that each of these $2\chi=O(1)$ submatrices have nonnegative rank $O(\sqrt n)$. This will imply that $\rank_+ M\leq O(\sqrt{n})$, which is equivalent to the desired statement $\xc(P)\leq O(\sqrt{n})$.

Fix $c\in \lbrace 1,\dots, \chi\rbrace$. Our goal is to show that $\rank_+ M[W_c,F_c]=O(\sqrt{n})$ and $\rank_+ M[V_c,F_c]=O(\sqrt{n})$.

For every point $a\in A_c$ let $F^a\su F_c$ be the set of facets that are encapsulated in the solid cap of radius $\eps$ centred at $a$. Since all points in $A_c$ have spherical distance at least $30\eps$ from each other, these sets $F^a$ are disjoint, so they form a partition of $F_c$. Furthermore, let $R$ be the constant in \cref{lem:shitov}.
 
First, we consider the matrix $M[W_c,F_c]$, which is somewhat
simpler to handle than $M[V_c,F_c]$. For each $a\in A_c$ we wish to apply \cref{lem:shitov} to the matrix $M[W_c,F^a]$. As in the statement of \cref{lem:shitov}, let $X$ be the solid cap of radius $\eps$ centred at $a$, and let $Y$ be the convex hull of all points on $S$ with spherical distance at least $5\varepsilon$ from $a$. Then, all facets in $F_a$ are encapsulated by $X$, and $W_c\su Y$. All the rows of $M[W_c,F^a]$ satisfy condition (1) in \cref{lem:shitov}, so we obtain that $\rank_+ M[W_c,F^a]\leq R$. Since the matrices $M[W_c,F^a]$, for $a\in A_c$, partition the matrix $M[W_c,F_c]$, it follows that $\rank_+ M[W_c,F_c]\le |A_c|\cdot R\le |A|\cdot R=O(\sqrt n)$.

It remains to consider the matrix $M[V_c,F_c]$. For every point $a\in A_c$ let $V^a\su V_c$ be the set of vertices lying in the solid cap of radius $5 \eps$ centred at $a$. Since all  points in $A_c$ have spherical distance at least $30\eps$ from each other, these sets $V^a$ partition $V_c$. We are assuming that each solid cap of radius $5\eps$ centred at some point $a\in A$ contains $O(\sqrt n)$ vertices, so we have $|V^a|=O(\sqrt n)$ for each $a\in A_c$. Furthermore, for any \emph{distinct} elements $a,a'\in A_c$, every vertex in $V^a$ is in the convex hull of all points on $S$ with spherical distance at least $25\eps$ from $a'$.

Let $N=\max_{a\in A_c}\vert V^a\vert=O(\sqrt n)$. Furthermore, let us fix a function $\phi:V_c\to \{1,\dots,N\}$ such that for each $a\in A_c$ the restriction $\phi|_{V^a}$ of $\phi$ to $V^a$ is a bijection $\phi|_{V^a}:V^a\to \{1,\dots,|V^a|\}$ (we can choose such a function $\phi$ by choosing bijections $V^a\to \{1,\dots,|V^a|\}$ separately for each $a\in A_c$, recalling that the sets $V^a$ form a partition of $V_c$). We can think of this function $\phi$ as a ``labelling'' that assigns each each vertex $v\in V^a$ a unique label in $\{1,\dots,|V^a|\}$. Now, for $i=1,\dots,N$ let us define a nonnegative vector $t^{(i)}=(t^{(i)}_f)_{f\in F_c}$, with entries indexed by facets $f\in F_c$. For every $a\in A_c$, and every facet $f\in F^a$, let us define the entry $t^{(i)}_f$ as follows.
If $i\leq \vert V^a\vert$, let $v$ be the unique vertex in $V^a$ with $\phi(v)=i$, and define $t^{(i)}_f=M_{v,f}$ to be the slack of the vertex $v$ with respect to the facet $f$. Otherwise, if $i> \vert V^a\vert$, define $t^{(i)}_f=0$.

Now, let $K$ be the $V_c\times F_c$ matrix defined by $K_{v,f}=M_{v,f}-t^{(\phi(v))}_f$ for all $v\in V_c$ and $f\in F_c$. In other words, $K$ is obtained from $M[V_c,F_c]$ by subtracting the vector $t^{(\phi(v))}$ from the row of $M[V_c,F_c]$ with index $v$, for each $v\in V_c$. The purpose of this definition is that for any $a\in A_c$, any vertex $v\in V^a$, and any facet $f\in F^a$, we have $K_{v,f}=M_{v,f}-t^{(\phi(v))}_f=0$. That is to say, for each $a\in A_c$ all entries of the submatrix $K[V^a, F^a]$ are zero.

\begin{claim}\label{claim-matrix-K}
The matrix $K$ has nonnegative entries and satisfies $\rank_+ K\leq \vert A_c\vert\cdot R$.
\end{claim}
\cref{claim-matrix-K} implies that there is a collection of $\vert A_c\vert\cdot R$ nonnegative vectors such that each row of $K$ can be written as a nonnegative linear combination of these vectors. But then, using these $\vert A_c\vert\cdot R$ vectors together with the $N$ vectors $t^{(1)},\dots,t^{(N)}$, we can obtain any row of $M[V_c,F_c]$ as a nonnegative linear combination. Thus, we obtain $\rank_+ M[V_c,F_c]\leq \vert A_c\vert\cdot R+N\leq \vert A\vert\cdot R+N=O(\sqrt n)$,
as desired. It only remains to prove \cref{claim-matrix-K}.

\begin{proof}[Proof of \cref{claim-matrix-K}]
We proceed in a similar way to the proof of the bound $M[W_c,F_c]=|A_c|\cdot R$ earlier in this section. Since the sets $F^a$ partition $F_c$, it suffices to prove that for each $a\in A_c$, the submatrix $K[V_c,F^a]$ is nonnegative and has nonnegative rank at most $R$. Since all the entries of $K[V^a,F^a]$ are zero, we actually only need to consider submatrices of the form $K[V_c\setminus V^a,F^a]$.

So let us fix $a\in A_c$. We will now apply \cref{lem:shitov} to show that $K[V_c\setminus V^a,F^a]$ has nonnegative entries and satisfies $\rank_+ K[V_c\setminus V^a,F^a]\leq R$. As in the statement of \cref{lem:shitov} applied with the point $a$ and the radius $5\eps$, let $X$ be the solid cap of radius $5\eps$ centred at $a$, and let $Y$ be the convex hull of all points on $S$ with spherical distance at least $25\eps$ from $a$. Each facet in $F^a$ is encapsulated by $X$, and since $V_c\setminus V^a$ is the union of the sets $V_{a'}$ for $a'\in A_c\setminus \lbrace a\rbrace$, all vertices in $V_c\setminus V^a$ are contained in $Y$.

Now, let us check that $K[V_c\setminus V^a,F^a]$ satisfies for each $v\in V_c\setminus V^a$ at least one of the conditions (1) and (2) in \cref{lem:shitov}. Note that for all $f\in F^a$ we have $K[V_c\setminus V^a,F^a]_{v,f}=K_{v,f}=M_{v,f}-t^{(\phi(v))}_f$. If $\phi(v)\leq \vert V^a\vert$, then there is a vertex $x_v\in V^{a}$ with $\phi(x_v)=\phi(v)$ and we have $t^{(\phi(v))}_f=M_{x_v,f}$ for all $f\in F^a$, so condition (1) holds. Otherwise, if $\phi(v)> \vert V^a\vert$, then $t^{(\phi(v))}_f=0$ for all $f\in F^a$, so condition (2) holds. Thus, all conditions of \cref{lem:shitov} are satisfied and we obtain that the matrix $K[V_c\setminus V^a,F^a]$ has nonnegative entries and satisfies $\rank_+ K[V_c\setminus V^a,F^a]\leq R$. This finishes the proof of \cref{claim-matrix-K}.\end{proof}

\section{Construction of the polyhedral lampshade \texorpdfstring{$\boldsymbol{Q}$}{Q}}
\label{sect-cone}

Here we prove \cref{lem:cap-polyhedron}. So let us fix $d\geq 2$. We start with an auxiliary construction (which determines the value of $R$ in \cref{lem:cap-polyhedron}).

\begin{fact}\label{fact:ball-polytope}
Let $B^{d-1}_1,B^{d-1}_{1/2}\su \RR^{d-1}$ be concentric $(d-1)$-dimensional balls, both centred at the origin, with radii $1$ and $1/2$ respectively. Let $S^{d-2}_1$ be the boundary of $B^{d-1}_1$ (so $S^{d-2}_1$ is a $(d-2)$-dimensional unit sphere). Then there is a convex polytope $P\subseteq B^{d-1}_1$ with vertices lying on $S_1^{d-2}$, whose interior contains the smaller ball $B^{d-1}_{1/2}$.
\end{fact}

\begin{proof}
We can choose $P$ to be the convex hull of some closely-spaced points on $S^{d-2}_1$, so that $P$ closely approximates $B^{d-1}_1$ (and therefore contains $B^{d-1}_{1/2}$). More precisely, let $\delta>0$ be small enough such that a solid cap of radius $\delta$  in the ball $B_1^{d-1}$ is disjoint from $B_{1/2}^{d-1}$, and consider a covering of the sphere $S_1^{d-2}$ by a finite collection of spherical caps of radius $\delta$. Then the polytope $P$ whose vertices are the centers of these spherical caps has the desired property.
\end{proof}

Let us fix a polytope $P$ as in \cref{fact:ball-polytope}, and let $R$ be twice the number of vertices in of $P$.

We are now ready to construct our convex set $Q$. Recall that $a\in S$ is a point on the unit sphere, and $X\su B$ is the solid cap of radius $\eps$ centred at $a$. Let $Z\su S$ be the subset of all points on the sphere whose spherical distance to $a$ is exactly $2\eps$. Then $Z$ is a $(d-2)$-dimensional sphere and its convex hull $\conv(Z)$ is a $(d-1)$-dimensional ball (given as the intersection of the unit ball $B\su \RR^d$ with some affine hyperplane $H_Z$). Denote the centre and radius of the $(d-1)$-dimensional ball $\conv(Z)$ by $b_Z$ and $r_Z$. Now, by rescaling and translating the polytope $P$ in \cref{fact:ball-polytope}, we can obtain a convex polytope $P_Z\su H_Z$ with $R/2$ vertices, such that all vertices of $P_Z$ are in $Z$, and such that $P_Z$ contains the $(d-1)$-dimensional ball in $H_Z$ centred at $b_Z$ with radius $r_Z/2$.

We can now define $Q$ as
\[Q:=\left\lbrace z+t(z-a)\mid z\in P_Z,\, t\in \RR_{\geq 0}\right\rbrace\]
In other words, $Q$ consists of all points obtained as follows. For any point $z\in P_Z$, we draw the line through $a$ and $z$ and consider all points on this line that are on the other side of $z$ from $a$ (these points form a ray starting at $z$ pointing away from $a$). The union of all these rays for all $z\in P_Z$ is the set $Q$.

It is not hard to see that $Q$ is indeed a convex set. Indeed, it is an (unbounded) $d$-dimensional polyhedron, given as an intersection of finitely many half-spaces: let $\mathcal{H}$ be the set consisting of the hyperplane $H_Z$ as well as, for each facet $f$ of the $(d-1)$-dimensional polytope $P_Z$, the hyperplane passing through $a$ and $f$. Then $Q$ is the intersection of finitely many (closed) half-spaces bounded by the hyperplanes in $\mathcal{H}$. Note that in particular all points of $Q$ lie in the (closed) half-space bounded by the hyperplane $H_Z$ not containing $a$. We now need to check that $Q$ satisfies conditions (i) to (iii) in \cref{lem:cap-polyhedron}.

Let us start by checking condition (iii). Consider a finite set of points $A\su Q$. All the points of $A$ lie in the (closed) half-space bounded by the hyperplane $H_Z$ and not containing $a$. We can now find a hyperplane $H_Z'$ parallel to $H_Z$, which is sufficiently far away from $H_Z$ such that all points in $A$ lie between the hyperplanes $H_Z$ and $H_Z'$ (or lie on $H_Z$ and $H_Z'$ themselves). Let $Q_A$ be the set of all points in $Q$ lying between the hyperplanes $H_Z$ and $H_Z'$ (or lying on $H_Z$ and $H_Z'$ themselves). Then $A\su Q_A$, so it suffices to show that $Q_A$ can be obtained as the convex hull of $R$ points in $Q$.

Indeed, consider the $R/2$ vertices of the polytope $P_Z$ (these vertices also lie in $Q_A\su Q$). Furthermore, for each vertex $z$ of $P_Z$, consider the intersection of the line through $a$ and $z$ with the hyperplane $H_Z'$ (this intersection is another point in $Q_A\su P$). By taking these intersection points for all $R/2$ vertices $z$ of  $P_Z$, we obtain $R/2$ additional points. All in all, this gives $R$ points in $Q_A\su Q$. It is not hard to see that the convex hull of these $R$ points is indeed the entire set $Q_A$. This establishes condition (iii).

In order to establish conditions (i) and (ii), the following lemma will be useful. Recall that $S$ is the boundary of the unit ball $B\su \RR^d$.

\begin{lem}\label{lem-hyperplane-ray}
Let $H\su \RR^d$ be a hyperplane intersecting the interior of the unit ball $B\su \RR^d$. Then $H$ cuts the ball $B$ into two solid caps. Let $q\in S$ be the centre of one of these two solid caps (this means $q$ is one of the two intersection points of $S$ with the line orthogonal to $H$ through the centre of the ball $B$). Suppose $U$ and $W$ are subsets of the unit sphere $S$ satisfying the following two assumptions.
\begin{itemize}
    \item[(a)] For each point $u\in U$ and each point $w\in W$ the spherical distance from $u$ to $q$ is at most the spherical distance from $w$ to $q$.
    \item[(b)] All points $w\in W$ lie on the opposite side of $H$ from the point $q$ (and do not lie on $H$ itself).
\end{itemize}
Let $u'\in \conv(U)$ and $w'\in \conv(W)$, and consider the ray $\left\lbrace w'+t(w'-u')\mid t\in \RR_{\geq 0}\right\rbrace$ along the line through $u'$ and $w'$ starting at $w'$ and pointing away from $u'$. Then this ray is disjoint from the hyperplane $H$ and lies on the other side of $H$ from the point $q$.
\end{lem}

\begin{proof}
Recall that the unit ball $B$ is centred at the origin. Therefore the hyperplane $H$ is given by an equation of the form $q\cdot x = d$ for some $d\in (-1,1)$ (note that $q$ is a normal vector for $H$). For any $x\in S$, the spherical distance $\theta\in [0,\pi]$ between $x$ and $q$ satisfies $\cos \theta= q\cdot x$. Assumption (a), and the fact that $\cos$ is a monotone decreasing function on the interval $[0,\pi]$, therefore imply that $q\cdot u\geq q\cdot w$ for all $u\in U$ and $w\in W$. Hence $q\cdot u'\geq q\cdot w'$ for all $u'\in \conv(U)$ and $w'\in \conv(W)$. Furthermore, assumption (b) means that we have $q\cdot w<d$ for all $w\in W$, and consequently $q\cdot w'<d$ for all $w'\in \conv(W)$. Now, let $u'\in \conv(U)$ and $w'\in \conv(W)$, and consider any point $x$ of the form $x=w'+t(w'-u')$ for some $t\in \RR_{\geq 0}$. Then
\[q\cdot x= q\cdot w'+t(q\cdot w'-q\cdot u')\leq q\cdot w'<d.\]
This shows that all points $x\in \left\lbrace w'+t(w'-u')\mid t\in \RR_{\geq 0}\right\rbrace$ are indeed on the other side of $H$ from the point $q$ (and not on $H$ itself).
\end{proof}

We take a moment to recall that $X\su B$ is the solid cap with radius $\eps$ centred at $a$, that $Y$ is the convex hull of the set of all points on $S$ with spherical distance at least $5\eps$ from $a$, and that $Z\su S$ is the set of all points on $S$ with spherical distance exactly $2\eps$ from $a$. In particular, note that $Z\cap X=\emptyset$, and that all points of $Z$ have spherical distance at least $\eps$ from all points in $X\cap S$.

We next check that $Q$ satisfies condition (i) using \cref{lem-hyperplane-ray}. Let $H\su \RR^d$ be an affine hyperplane such that $H$ intersects the interior of $B$ and such that $H\cap B\su X$. Then one of the two solid caps into which $H$ divides the ball $B$ (namely, the smaller of the two) is a subset of $X$. Let $q\in X\cap S$ be the centre of this solid cap cut out by $H$. We wish to apply \cref{lem-hyperplane-ray} with $U=\lbrace a\rbrace$ and $W=Z$ (and $q$ as we have just defined). Since $q\in X\cap S$, the spherical distance between $q$ and $a$ is at most $\eps$, and the spherical distance between $q$ and any point in $Z$ is at least $\eps$, so assumption (a) in \cref{lem-hyperplane-ray} is satisfied. Furthermore, assumption (b) is satisfied basically by definition: $H$ divides $B$ into two solid caps, one of which contains $q$, and the other of which contains all points in $B\setminus X$, including all points in $W=Z$ (and no point of $W=Z$ lies on $H$ itself). So, \cref{lem-hyperplane-ray} says that for any point $z\in P_Z\su \conv(Z)=\conv(W)$, the ray $\left\lbrace z+t(z-a)\mid t\in \RR_{\geq 0}\right\rbrace$ is disjoint from $H$. Since $Q$ is, by definition, the union of all these rays, $Q$ is disjoint from $H$. Thus, $Q$ satisfies condition (i) in \cref{lem:cap-polyhedron}.

It remains to check that $Q$ satisfies condition (ii). Recall that $Q$ can be expressed as the intersection of finitely many half-spaces, and that $\mathcal H$ is the collection of hyperplanes bounding these half-spaces. Note that the origin $b=0$ (which is the centre of the ball $B$) lies in the interior of $Q$. Indeed, recall that $b_Z$ is the centre of the $(d-1)$-dimensional ball formed by $\conv(Z)$, so by construction it lies in the interior of the $(d-1)$-dimensional polytope $P_Z$. Since $\eps<\pi/5<\pi/4$, this point $b_Z$ lies on the line segment between the points $a$ and $b$, so $b$ is of the form $b_Z+t(b_Z-a)$ for some $t>0$. So $b$ is indeed in the interior of $Q$, and consequently $Q$ is the intersection of all the (closed) half-spaces containing $b$ that are bounded by a hyperplane $H\in \mathcal{H}$.

\begin{claim}\label{claim-cone-geometry}
Let $H\in \mathcal{H}$, and consider the two solid caps into which the hyperplane $H$ cuts the ball $B$. Let $q\in S$ be the centre of the solid cap not containing the centre $b$ of the ball (i.e.\ the smaller cap). Then $q$ has spherical distance at most $2\eps$ from $a$.
\end{claim}

Before proving \cref{claim-cone-geometry} at the end of this subsection, we show how to use the claim to check that $Q$ satisfies condition (ii). Consider any hyperplane $H\in \mathcal{H}$, let $q$ be as in \cref{claim-cone-geometry}, and note that then $q$ and $b$ lie on opposite sides of the hyperplane $H$. We wish to apply \cref{lem-hyperplane-ray} with the sets $U=X\cap S\su S$ and $W=Y\cap S\su S$, so we need to check that these sets satisfy assumptions (a) and (b).

By \cref{claim-cone-geometry}, the point $q$ has spherical distance at most $2\eps$ from $a$. Consequently, $q$ has spherical distance at most $3\eps$ from every point in $X\cap S$. Furthermore, because all points in $Y\cap S$ have spherical distance at least $5\eps$ from $a$, we can conclude that $q$ has spherical distance at least $3\eps$ from every point in $Y\cap S$. This verifies assumption (a).

For assumption (b), recall that $q$ is the centre of a solid cap bounded by $H$. We claim that the common spherical distance between $q$ and all the points in $H\cap S$ (i.e.\ the radius of this cap) is at most $2\eps$. If $H=H_Z$, then $q=a$ and $H\cap S=Z$, so the spherical distance is exactly $2\eps$ by construction. Recall that each of the hyperplanes in $\mathcal H\setminus \lbrace H_Z\rbrace$ passes through $a$ and some facet $f$ of $P_Z$. So, if $H\ne H_Z$, then $H\cap S$ contains the point $a$, and \cref{claim-cone-geometry} implies that the spherical distance between $q$ and $H\cap S$ is at most $2\eps$, as claimed.

Together with our observation above that $q$ has spherical distance at least $3\eps$ from all points in $W=Y\cap S$, this implies that $W=Y\cap S$ and $q$ lie on opposite sides of the hyperplane $H$. This verifies assumption (b). Thus, by \cref{lem-hyperplane-ray}, for all points $x\in X=\conv(X\cap S)$ and $y\in Y=\conv(Y\cap S)$ the ray $\left\lbrace y+t(y-x)\mid t\in \RR_{\geq 0}\right\rbrace$ lies on the opposite side of $H$ from $q$ (i.e.\ on the same side of $H$ as $b$). We have proved that for all $x\in X$ and $y\in Y$, the ray $\left\lbrace y+t(y-x)\mid t\in \RR_{\geq 0}\right\rbrace$ lies inside each of the half-spaces defining $Q$ (bounded by the hyperplanes $H\in \mathcal{H}$), so the ray lies inside $Q$ itself. This shows that $Q$ satisfies condition (ii) in \cref{lem:cap-polyhedron}.

It remains to prove \cref{claim-cone-geometry}.

\begin{proof}[Proof of \cref{claim-cone-geometry}]
First note that the case $H=H_Z$ is immediate, because in this case $q=a$ ($H_Z$ is the hyperplane passing through the set of all points at spherical distance $2\eps$ from $a$, and $2\eps<\pi/2$, so $a$ is the centre of the smaller of the two caps that $H_Z$ cuts $B$ into). So let us from now on assume that $H\neq H_Z$, meaning that $H$ passes through $a$ and one of the facets of the $(d-1)$-dimensional polytope $P_Z$.

Let $A$ be the two-dimensional plane passing through $b$, $a$ and $q$. Recall that $b_Z$ and $r_Z$ are the centre and the radius of the $(d-1)$-dimensional ball $\conv(Z)$. Also recall that $b_Z$ lies on the line segment between $a$ and $b$, so $b_Z$ lies in $A$ as well. We will intersect all relevant objects with $A$, which will make everything much easier to visualise. Since the line $bq$ is orthogonal to $H$, and the line $ab$ is orthogonal to $H_Z$, the intersections $H\cap A$ and $H_Z\cap A$ are both lines. Each of these lines intersects the circle $S\cap A$ in two points (and in the former case, one of these points is $a$). Let $z$ and $z'$ be the intersections of $H_Z\cap A$ with $S\cap A$ (then $z$ and $z'$ are the endpoints of the line segment $A\cap \conv(Z)$, and $b_Z$ is the midpoint of this segment), and note that $z,z'\in Z$. Furthermore, let $s$ be the intersection of $H\cap A$ with $S\cap A$ other than $a$. Recall that $H$ passes through a facet of $P_Z$, and that $P_Z$ was chosen such that its interior contains the $(d-1)$-dimensional ball in $H_Z$ centred at $b_Z$ with radius $r_Z/2$. Hence, this $(d-1)$-dimensional ball is disjoint from $H\cap H_Z$. Therefore the intersection point $p$ between the lines $H\cap A$ and $H\cap H_Z$ lies outside the line segment represented in bold in \cref{fig-plane-Q}, between the midpoint of $b_Z$ and $z$ and the midpoint of $b_Z$ and $z'$.
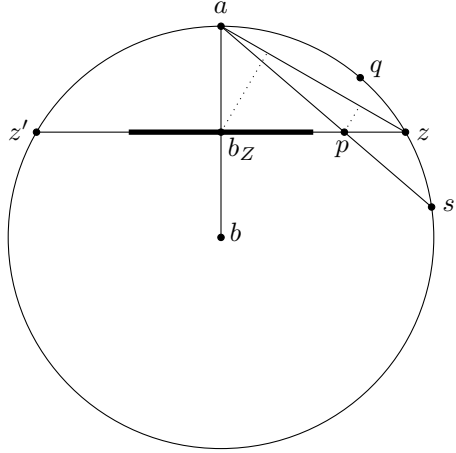
\begin{figure}
    \centering
\begin{tikzpicture}[x=3.5cm,y=3.5cm,scale=0.8]
\clip(-1.1,-1.05) rectangle (1.1,1.15);
\draw  (0,0) circle (1);
\draw (0,1)-- (0,0);
\draw (-0.866926409059089,0.4984361556668147)-- (0.8669264090590889,0.4984361556668148);
\draw (0,1)-- (0.9895925153740783,0.1438980664067608);
\draw (0,1)-- (0.8669264090590889,0.4984361556668148);
\draw [line width=2pt] (-0.4334632045295445,0.49843615566681476)-- (0.43346320452954445,0.4984361556668148);
\draw [dotted] (0,0.4984361556668148)-- (0.21740947124082008,0.8742168550288583);
\draw [dotted] (0.57977187862561,0.4984361556668148)-- (0.651785043726563,0.6229070786236873);
\draw [fill] (0,0) circle (1.5pt);
\draw (0.06888320624343089,0.026397169055649438) node {$b$};
\draw [fill] (0,1) circle (1.5pt);
\draw (0,1.0884273517115737) node {$a$};
\draw [fill] (0.8669264090590889,0.4984361556668148) circle (1.5pt);
\draw (0.95,0.49) node {$z$};
\draw [fill] (-0.866926409059089,0.4984361556668147) circle (1.5pt);
\draw (-0.95,0.51) node {$z'$};
\draw [fill] (0.9895925153740783,0.1438980664067608) circle (1.5pt);
\draw (1.07,0.15) node {$s$};
\draw [fill] (0,0.4984361556668148) circle (1.5pt);
\draw (0.1,0.42) node {$b_Z$};
\draw [fill] (0.57977187862561,0.4984361556668148) circle (1.5pt);
\draw (0.5694302185607726,0.42) node {$p$};
\draw [fill] (0.6542560407031941,0.7562731207727671) circle (1.5pt);
\draw (0.73,0.8) node {$q$};
\end{tikzpicture}
\caption{A view of the two-dimensional plane $A$ through $b$, $a$ and $q$. }
\label{fig-plane-Q}
\end{figure}

For two points on the circle $S\cap A$, the \emph{arc-distance} between them is the length of the shorter of the two circular arcs between the two points. Our goal is to show that the arc-distance between $a$ and $q$ is at most $2\varepsilon$. Note that $q$ is the midpoint of the (shorter) arc between $a$ and $s$ (since on the sphere $S$, the point $q$ is the centre of the smaller solid cap bounded by $H$). It therefore suffices to show that the arc-distance between $a$ and $s$ is at most $4\eps$.

To show this, first note that the arc-distances from $a$ to $z$ and to $z'$ are both $2\eps$ (since $z,z'\in Z$). Let us now briefly consider the possibility\footnote{It turns out that this is actually impossible, but treating this case is a bit simpler than showing that it is impossible.} that $a$ and $s$ lie on the same side of the chord $zz'$ of the circle $S\cap A$ (contrary to the illustration in \cref{fig-plane-Q}). In this case, the arc-distance between $a$ and $s$ is at most the arc-distance from $a$ to $z$ and to $z'$, and therefore at most $2\eps\le 4\eps$, as desired.

So, we may assume that $a$ and $s$ lie on different sides of the chord $zz'$, meaning that $p$ lies on the chord $zz'$. To disambiguate between $z$ and $z'$, let us assume that $p$ is closer to $z$ than to $z'$. Then $p$ lies on the line segment between $b_Z$ and $z$, and is closer to $z$ than to $b_Z$ (since $p$ lies outside the line segment between the midpoint of $b_Z$ and $z$ and the midpoint of $b_Z$ and $z'$).

Now, recall that $z$ and $z'$ both have (the same) arc-distance $2\eps$ from $a$, and $b_Z$ is the midpoint of the chord $zz'$, so the triangle $ab_Zz$ has a right angle at $b_Z$. Hence the orthogonal projection of $b_Z$ onto the line $az$ lies in the interior of the segment $az$. The orthogonal projection of $p$ onto the line $az$ lies between the projection of $b_Z$ and the point $z$, and it is closer to the point $z$ than to the projection of $b_Z$. Consequently, the projection of $p$ is closer to $z$ than to $a$. This means that $p$ lies on the same side of the perpendicular bisector of $az$ as $z$. Hence the distance of $p$ to $z$ is smaller than the distance of $p$ to $a$. Consequently, in the triangle $pza$ the angle $\angle paz$ is smaller than the angle $\angle azp$. In other words, the angle $\angle saz$ is smaller than the angle $\angle azz'$. Hence the arc between $s$ and $z$ is shorter than the arc between $a$ and $z'$. However, the latter arc has length exactly $2\eps$. Thus, the arc-distance between $z$ and $s$ is at most $2\eps$. Since the arc-distance between $a$ and $z$ equals $2\eps$, this implies that the arc-distance between $a$ and $s$ is at most $4\eps$, as desired.
\end{proof}

\section{Preparations for the proof for cyclic polygons}
\label{sec:preparations-cyclic}

In the proof of \cref{thm:cyclic-polygon}, we will use the following lemma due to Shitov~\cite[Lemma 3.1]{Sh14v1}.

\begin{lem}
\label{lem:shitov-original}
Let $P$ be a polygon, and let $V$ and $F$ be the sets of vertices and facets (edges) of $P$, respectively. Let $X\su V$ be a set of consecutive vertices of $P$, and let $F'\su F$ be the set of facets of $P$ with both endpoints in $X$.

Let $M$ be a slack matrix of the polytope $P$, with rows indexed by $V$ and columns indexed by $F$. Let $M'$ be a matrix with rows indexed by $V\setminus X$ and columns indexed by $F'$, such that the following condition holds for each vertex $v\in V\setminus X$:
\begin{itemize}
    \item[($\star$)] there are a vertex $x_v\in X$ and real numbers $\alpha_v>0$ and $\beta_v\geq 0$ such that $M_{v,f}'=\alpha_v M_{v,f}-\beta_v M_{x_v,f}$ for all $f\in F'$.
\end{itemize}
Then, if all the entries of $M'$ are nonnegative, we have $\rank_+ M'\le 8$.
\end{lem}

Shitov's original lemma~\cite[Lemma 3.1]{Sh14v1} is actually more general: in condition ($\star$) it allows $M'_{v,f}$ to be a linear combination of $M_{v,f}$ and all of the entries $M_{x,f}$ with $x\in X$ (not just one particular entry $M_{x_v,f}$), where again the coefficients only depend on $v$. However, since this more general form of the lemma requires more complicated notation, we only stated the special case\footnote{Actually, strictly speaking \cref{lem:shitov-original} is not quite a special case of \cite[Lemma 3.1]{Sh14v1}, and a tiny bit of deduction is required. The statement of Shitov's original lemma~\cite[Lemma 3.1]{Sh14v1} is not written to allow an arbitrary coefficient $\alpha_v>0$ in condition ($\star$), it only allows $\alpha_v=1$. However, we can simply rescale every row of $M'$ by the reciprocal of the corresponding coefficient $\alpha_v$, to put us in the setting of \cite[Lemma 3.1]{Sh14v1}. These rescalings do not affect the nonnegative rank of $M'$.} that we need for the proof of \cref{thm:cyclic-polygon}.

The reader may want to compare \cref{lem:shitov-original} with \cref{lem:shitov}. In \cref{lem:shitov}, we impose some conditions, and deduce that a certain matrix $M'$ has nonnegative entries and bounded nonnegative rank. Here, in \cref{lem:shitov-original}, one of the conditions we impose is that $M'$ has nonnegative entries, and we deduce that $M'$ has bounded nonnegative rank.

Note that for the proof of \cref{thm:cyclic-polygon} we may assume without loss of generality that all vertices of the polygon $P$ lie on the unit circle $\Gamma$ around the origin. For two points $x,y\in \Gamma$, let us define the \emph{arc-distance} between $x$ and $y$ to be the length of the shorter arc between $x$ and $y$ along the circle $\Gamma$. In the proof of \cref{thm:cyclic-polygon}, we will use the assumption that $P$ is a cyclic polygon (with vertices on $\Gamma$) by applying the following lemma. This will be the only place where we use the assumption that $P$ is cyclic.

\begin{lem}\label{lem:circle-geometry}
Let $P$ be a polygon all of whose vertices lie on the unit circle $\Gamma$. Let $V$ and $F$ be the sets of vertices and facets (edges) of $P$, and let $M$ be a slack matrix of the polytope $P$, with rows indexed by $V$ and columns indexed by $F$.

Let $X\su \Gamma$ be an arc of $\Gamma$ of length $\eps>0$, and let $Y\su \Gamma$ be the set of all points on $\Gamma$ with arc-distance at least $5\eps$ from every point of the arc $X$.

Now suppose that $v\in V$ is a vertex on the arc $X$ and that $f\in F$ is a facet both of whose endpoints are on the arc $X$. Furthermore, suppose that $w\in V$ is a vertex with $w\in Y$ and that $g\in F$ is a facet both of whose endpoints are in the set $Y$. Then $M_{v,f}M_{w,g}\leq M_{v,g}M_{w,f}$.
\end{lem}

We postpone the proof of \cref{lem:circle-geometry} to \cref{section-geometry-circle}. Roughly speaking, the idea is as follows. Since $v$ and the endpoints of $f$ are on the arc $X$, but $w$ and the endpoints of $g$ are relatively far away from $X$, the slack $M_{v,f}$ is significantly smaller than the slacks $M_{w,f}$ and $M_{v,g}$. The slack $M_{w,g}$ may be large, but we will argue using the triangle inequality that then also one of the slacks $M_{v,g}$ and $M_{w,f}$ needs to be large.

Finally, we will need the following lemma about matrices. Very roughly speaking, this lemma states that if a matrix satisfies certain inequalities between products of its entries, then one can rescale the rows of the matrix in such a way that certain entries are larger than certain other entries. In the proof of \cref{thm:cyclic-polygon}, we will apply this lemma to certain submatrices of the slack matrix. When doing so, we will use \cref{lem:circle-geometry} to show that these submatrices of the slack matrix satisfy the assumptions of \cref{lem:matrix-rescaling}.

\begin{lem}\label{lem:matrix-rescaling}
Let $M$ be a nonnegative matrix with rows indexed by $\lbrace 1,\dots,m\rbrace$ and columns indexed by some set $S$. Let $S=S_1\cup\dots\cup S_m$ be a partition of $S$ into non-empty subsets. Suppose that for each $j\in \lbrace 1,\dots,m\rbrace$ and each $s\in S\setminus S_j$, we have $M_{j,s}>0$. Furthermore suppose that for each $j\in \lbrace 1,\dots,m\rbrace$, each $k\in \lbrace 1,\dots,j-1\rbrace$, each $s\in S_j$ and each $t\in S_1\cup \dots\cup S_{j-1}$, we have $M_{j,s}M_{k,t}\leq M_{j,t}M_{k,s}$. Then there exist positive real numbers $\alpha_1,\dots, \alpha_m$ such that we have $\alpha_j M_{j,s}\leq \alpha_k M_{k,s}$ whenever $j,k\in \lbrace 1,\dots,m\rbrace$ and $s\in S_j$.
\end{lem}

\begin{proof}
We prove the lemma by induction on $m$. The case $m=1$ is trivial (we can take any $\alpha_1>0$ and the inequality $\alpha_1 M_{1,s}\leq \alpha_1 M_{1,s}$ is trivially satisfied for each $s\in S_1$).

Let us now assume that $m\geq 2$ and that the lemma is already proved for $m-1$. Then (by ignoring the row with index $m$ and the columns with indices in $S_m$) we can find positive real numbers $\alpha_1,\dots, \alpha_{m-1}$ such that we have $\alpha_j M_{j,s}\leq \alpha_k M_{k,s}$ whenever $j,k\in \lbrace 1,\dots,m-1\rbrace$ and $s\in S_j$.

Now we need to find $\alpha_m>0$ such that the desired inequality $\alpha_j M_{j,s}\leq \alpha_k M_{k,s}$ for $j,k\in \lbrace 1,\dots,m\rbrace$ and $s\in S_j$ also holds if $j$ or $k$ are equal to $m$. Note that for $j=k=m$ the inequality $\alpha_m M_{m,s}\leq \alpha_m M_{m,s}$ is automatically satisfied for all $s\in S_m$.

Hence it suffices to find $\alpha_m>0$, such that both of the following conditions are satisfied:
\begin{equation}\label{eq:ineq-matrix-m-k}
\alpha_m M_{m,s}\leq \alpha_k M_{k,s}
\quad \text{ for all }k\in \lbrace
1,\dots,m-1\rbrace\text{ and }s\in S_m,
\end{equation}
\begin{equation}\label{eq:ineq-matrix-m-j}
\alpha_j M_{j,s}\leq \alpha_m M_{m,s}\quad \text{ for all }j\in \lbrace 1,\dots,m-1\rbrace\text{ and }s\in S_j.
\end{equation}

Let us first consider the case that $M_{m,s}=0$ for all $s\in S_m$. Then we can simply choose $\alpha_m>0$ large enough such that $\alpha_m M_{m,s}\geq \alpha_j M_{j,s}$ for all $j\in \lbrace 1,\dots,m-1\rbrace$ and all $s\in S_j$ (recall that $M_{m,s}>0$ for $s\in S\setminus S_m$). This satisfies \cref{eq:ineq-matrix-m-j}, and note that \cref{eq:ineq-matrix-m-k} is automatically satisfied if $M_{m,s}=0$ for all $s\in S_m$.

So we may from now on assume that $M_{m,s}>0$ for at least one choice of $s\in S_m$. Let us now define $\alpha_m$ by
\[\alpha_m=\min_{\substack{s\in S_m\\M_{m,s}\neq 0}}\ \min_{k\in \lbrace 1,\dots,m-1\rbrace}\frac{\alpha_k M_{k,s}}{M_{m,s}}.\]
Note that $\alpha_m$ is a well-defined real number, and we have $\alpha_m>0$ since $M_{k,s}>0$ for all $k\in \lbrace 1,\dots,m-1\rbrace$ and $s\in S_m$ (and $\alpha_1,\dots,\alpha_{m-1}>0$).

By the definition of $\alpha_m$, condition \cref{eq:ineq-matrix-m-k} is satisfied (note that it is automatically satisfied for those $s\in S_m$ with $M_{m,s}=0$). It remains to check \cref{eq:ineq-matrix-m-j}.
 
Let $j\in \lbrace 1,\dots,m-1\rbrace$ and $t\in S_j$. We need to show that $\alpha_j M_{j,t}\leq \alpha_m M_{m,t}$. By the definition of $\alpha_m$, we can find $s\in S_m$ and $k\in \lbrace 1,\dots,m-1\rbrace$ such that $\alpha_m M_{m,s}=\alpha_k M_{k,s}$. By applying the assumption of the lemma to $m$, $k$, $s$ and $t$, we obtain $M_{m,s}M_{k,t}\leq M_{m,t}M_{k,s}$, and therefore
\[\alpha_m M_{m,s}\cdot \alpha_k M_{k,t}\leq \alpha_m M_{m,t}\cdot \alpha_k M_{k,s}.\]
Using $\alpha_m M_{m,s}=\alpha_k M_{k,s}$, this implies $\alpha_k M_{k,t}\leq \alpha_m M_{m,t}$. But by the choice of $\alpha_1,\dots,\alpha_{m-1}$ in the induction hypothesis we also have $\alpha_j M_{j,t}\leq \alpha_k M_{k,t}$ (recall that $t\in S_j$). We conclude that $\alpha_j M_{j,t}\leq \alpha_k M_{k,t}\leq \alpha_m M_{m,t}$, as desired.
\end{proof}

\section{Upper-bounding the extension complexity of cyclic polygons}
\label{sec:cyclic}

In this section, we prove \cref{thm:cyclic-polygon}. Let $P$ be a cyclic polygon with $n$ vertices, and let $V$ and $F$ be its sets of vertices and facets (edges). Then $|V|=|F|=n$. By rescaling and translating $P$, we may assume without loss of generality that all vertices of $P$ lie on the unit circle $\Gamma$ around the origin. We may also assume that $n\geq 24^2=576$, since for $n<24^2$ we trivially have $\xc(P)\leq n<24\sqrt{n}$.

We can divide the $n$ facets of $P$ into $\lceil \sqrt{n}\,\rceil$ blocks of consecutive facets, such that each of these blocks consists of at most $\lceil \sqrt{n}\,\rceil$ facets. For each of these blocks, we obtain a (closed) arc $X$ of $\Gamma$ from the first vertex of the first facet in the block to the last vertex of the last facet in the block. Note that then each facet of the block has both endpoints in $X$, and in total the arc $X$ contains at most $\lceil \sqrt{n}\,\rceil +1$ vertices of $P$.

In this way, we obtain a collection $\mathcal{X}$ of arcs $X\su \Gamma$, such that $\vert \mathcal{X}\vert\leq \lceil \sqrt{n}\,\rceil$, the arcs $X\in\mathcal{X}$ are disjoint apart from their endpoints, each arc $X\in \mathcal{X}$ contains at most $\lceil \sqrt{n}\,\rceil +1$ vertices of $P$, and for each facet $f\in F$ there is exactly one arc $X\in \mathcal{X}$ such that both endpoints of $f$ are in $X$. Let us say that the facet $f$ \emph{belongs to} this arc $X$.

Note that the arcs in $\mathcal{X}$ may have different lengths. Recall that for any two points $x,y\in \Gamma$, we defined the arc-distance between $x$ and $y$ to be the length of the shorter arc between $x$ and $y$ along the circle $\Gamma$.

\begin{defn}
We say that two arcs $X, X'\in \mathcal{X}$ of lengths $\eps$ and $\eps'$ are \emph{well-separated} if the arc-distance between any two points $x\in X$ and $x'\in X'$ is at least $5\min\{\eps,\eps'\}$.
\end{defn}

\begin{claim}\label{claim:arc-colouring-counting}
For every arc $X\in \mathcal{X}$, there exist at most $13$ arcs $X'\in \mathcal{X}$ which are not well-separated from $X$ and are at least as long as $X$.
\end{claim}
\begin{proof}
Denote the length of the arc $X$ by $\eps$. Let $Y\su \Gamma$ be the set of points on the circle $\Gamma$ which have arc-distance at most $6\eps$ from some point of $X$. Note that then the set $Y$ is either an arc of $\Gamma$ of length $13\eps$, or all of $\Gamma$ (the second case occurs if $13\eps\geq 2\pi$). In either case, the length of $Y$ is at most $13\eps$.

If $X'\in \mathcal{X}$ is not well-separated from $X$ and has length at least $\eps$, then there must be points $x\in X$ and $x'\in X'$ of arc-distance less than $5\eps$. But then the intersection $X'\cap Y$ contains an entire arc of length $\eps$. Since the different arcs $X'\in \mathcal{X}$ are disjoint apart from their endpoints, there can be at most $13$ such arcs $X'\in \mathcal{X}$.
\end{proof}

\begin{claim}\label{claim:arc-colouring}
We can colour the elements of $\mathcal{X}$ with $14$ colours in such a way that any two arcs $X, X'\in \mathcal{X}$ of the same colour are well-separated.
\end{claim}
\begin{proof}
Let us order the arcs $X\in \mathcal{X}$ by decreasing length. Going through the arcs $X\in \mathcal{X}$ one by one in this order, we can now find the desired colouring greedily by assigning each arc $X\in \mathcal{X}$ a colour which is different from the colours of the previously coloured arcs $X'\in \mathcal{X}$ from which $X$ is not well-separated. Indeed, by \cref{claim:arc-colouring-counting}, for every $X\in \mathcal{X}$ there are at most $13$ such arcs $X'\in \mathcal{X}$ (note that all of the previously coloured arcs are at least as long as $X$).
\end{proof}

Let us colour the arcs of $\mathcal{X}$ with $14$ colours as in \cref{claim:arc-colouring}. For $c=1,\dots,14$, let $\mathcal{X}_c\su \mathcal{X}$ be the collection of arcs $X\in \mathcal{X}$ of colour $c$ (so $\mathcal{X}=\mathcal{X}_1\cup\dots\cup \mathcal{X}_{14}$ is a partition of $\mathcal{X}$). Then we obtain a partition $F=F_1\cup \dots\cup F_{14}$ of the facets of $P$, where for each $c=1,\dots,14$ we let $F_c$ be the set of facets belonging to some arc $X\in \mathcal{X}_c$ (i.e.\ to some arc of colour $c$).

Let $M$ be a slack matrix of the polygon $P$, with rows indexed by $V$ and columns indexed by $F$. We partition the matrix $M$ into $14$ submatrices $M[V,F_1],\dots,M[V,F_{14}]$, where for subsets $V'\su V$ and $F'\su F$, by $M[V',F']$ we denote the $V'\times F'$ submatrix of $M$ containing the slacks between vertices in $V'$ and facets in $F'$. It suffices to show that for each $c=1,\dots,14$ we have
\begin{equation}
\rank_+ M[V,F_c]\leq 8\vert \mathcal{X}_c\vert+\lceil\sqrt{n}\,\rceil +1.
\label{eq:bound-M-c-cyclic}
\end{equation}
Indeed, from \cref{eq:bound-M-c-cyclic} we obtain
\[\rank_+ M\leq \sum_{c=1}^{14} \rank_+ M[V,F_c]\leq \sum_{c=1}^{14}\left(8\vert \mathcal{X}_c\vert+\lceil\sqrt{n}\,\rceil +1\right)\leq 8\vert \mathcal{X}\vert+14\lceil\sqrt{n}\,\rceil +14
\leq 22\sqrt{n}+36,\]
which implies (by our assumption $n\geq 24^2$) that $\xc(P)=\rank_+ M\leq 24\sqrt{n}$, as desired.

So, let us from now on fix some $c\in \lbrace 1,\dots,14\rbrace$. For each $X\in \mathcal{X}_c$, let $V^X=V\cap X$ and let $F^X$ be the set of all facets of $P$ belonging to $X$. Let $V_c\su V$ be the set of all vertices in an arc $X\in\mathcal X_c$ (i.e.\ the union of all the sets $V^X$ for $X\in \mathcal{X}_c$). Since the arcs in $\mathcal X_c$ are well-separated and in particular disjoint, the sets $V^X$ for $X\in \mathcal{X}_c$ partition $V_c$ and the sets $F^X$ for $X\in \mathcal{X}_c$ partition $F_c$.

Let $N=\lceil\sqrt n\,\rceil+1$, and note that by the definition of the arcs in $\mathcal{X}$, we have $\vert V^X\vert\leq \lceil \sqrt{n}\,\rceil +1=N$ for each $X\in \mathcal{X}_c$. Let us fix a function $\phi:V_c\to \{1,\dots,N\}$ such that for each $X\in \mathcal{X}_c$ the restriction $\phi|_{V^X}$ of $\phi$ to $V^X$ is a bijection $\phi|_{V^X}:V^X\to \{1,\dots,|V^X|\}$ (we can choose such a function $\phi$ by choosing bijections $V^X\to \{1,\dots,|V^X|\}$ separately for each $X\in \mathcal{X}_c$). We can think of this function $\phi$ as a ``labelling'' that assigns each each vertex $v\in V^X$ a unique label in $\{1,\dots,|V^X|\}$.

\begin{claim}\label{claim:desired-scaling-matrix}
We can find positive real numbers $\alpha_v>0$ for all $v\in V_c$, such that the following holds. For any arc $X\in \mathcal{X}_c$, any facet $f\in F^X$, and any vertices $v\in V_c$ and $w\in V^X$ with $\phi(v)=\phi(w)$, we have $\alpha_v M_{v,f}\geq \alpha_w M_{w,f}$.
\end{claim}

We defer the proof of \cref{claim:desired-scaling-matrix} until later in this section. For each $i=1,\dots,N$, we now define a nonnegative vector $t^{(i)}$ with entries indexed by $F_c=\bigcup_{X\in \mathcal{X}_c}F^X$. For every $X\in \mathcal{X}_c$, and every facet $f\in F^X$, we define the entry $t^{(i)}_f$ as follows. If $i>\vert V^X\vert$, define $t^{(i)}_f=0$. Otherwise, i.e.\ if $i\leq \vert V^X\vert$, let $v$ be the unique vertex in $V^X$ with $\phi(v)=i$ and define $t^{(i)}_f=\alpha_v M_{v,f}$.

Let $K$ be the $V\times F_c$ matrix defined by letting $K_{v,f}=\alpha_v M_{v,f}-t^{(\phi(v))}_f$ for all $v\in V_c$ and $f\in F_c$, and letting $K_{v,f}=M_{v,f}$ for all $v\in V\setminus V_c$ and $f\in F_c$. In other words, we obtain $K$ from $M[V,F_c]$ by first scaling the rows with indices in $V_c$ with the factors $\alpha_v$ as in \cref{claim:desired-scaling-matrix}, and then subtracting $t^{(\phi(v))}$ from each row corresponding to a vertex $v\in V_c$. The purpose of these subtractions is to ensure that for any $X\in \mathcal X_c$, any vertex $v\in V^X\su V_c$ and any facet $f\in F^X$, we have $K_{v,f}=\alpha_v M_{v,f}-t^{(\phi(v))}_f=\alpha_v M_{v,f}-\alpha_v M_{v,f}=0$. That is to say, for each $X\in \mathcal{X}_c$, all entries of the submatrix $K[V^X,F^X]$ are zero.

\begin{claim}\label{claim:matrix-K-cyclic-polygon}
The matrix $K$ has nonnegative entries and satisfies $\rank_+ K\leq 8\vert \mathcal{X}_c\vert$.
\end{claim}

We also defer the proof of \cref{claim:matrix-K-cyclic-polygon} until later in this section. It is now straightforward to deduce \cref{eq:bound-M-c-cyclic}. Indeed, by \cref{claim:matrix-K-cyclic-polygon} there is a collection of $8\vert \mathcal{X}_c\vert$ nonnegative vectors such that each row of the matrix $K$ can be written as a nonnegative linear combination of these vectors. Using these $8\vert \mathcal{X}_c\vert$ vectors together with the $N$ vectors $t^{(1)},\dots,t^{(N)}$, we can obtain any row of the matrix $M[V,F_c]$ as a nonnegative linear combination. To see this, recall that for every $v\in V\setminus V_c$ the row of $M[V,F_c]$ with index $v$ is identical to the row of $K$ with index $v$. Also, for every $v\in V_c$, the row of $M[V,F_c]$ with index $v$ can be obtained from the row of $K$ with index $v$ by adding the vector $t^{(\phi(v))}$ and afterwards scaling by $\alpha_v^{-1}$ (recall that $\alpha_v>0$). We conclude that $\rank_+ M[V,F_c]\leq 8\vert \mathcal{X}_c\vert+N= 8\vert \mathcal{X}_c\vert+\lceil\sqrt{n}\,\rceil+1$,
as desired. It remains to prove \cref{claim:desired-scaling-matrix,claim:matrix-K-cyclic-polygon}.

\begin{proof}[Proof of \cref{claim:desired-scaling-matrix}]
First, note that we demand an inequality involving $\alpha_v$ and $\alpha_w$ only if $\phi(v)=\phi(w)$. In other words, we can find the real numbers $\alpha_v>0$ for $v\in V_c$ separately for each of the labels $\phi(v)\in \lbrace 1,\dots,N\rbrace$.

Fix some $i\in \lbrace 1,\dots,N\rbrace$. Consider all arcs in $X\in \mathcal{X}_c$ satisfying $\vert V^X\vert\geq i$ (i.e.\ all arcs containing a vertex labelled $i$), and order these arcs as $X(1), X(2),\dots, X(m)$ in order of decreasing length. For each $j=1,\dots,m$, let $v(j)$ be the unique vertex in $V^{X(j)}$ with $\phi(v(j))=i$. Note that then $\{v(1),\dots,v(m)\}$ is the set of all vertices in $V_c$ with label $\phi(v)=i$. Our goal is to find positive real numbers $\alpha_{v(1)},\dots,\alpha_{v(m)}$, such that we have $\alpha_{v(j)}M_{v(j),f}\le \alpha_{v(k)} M_{v(k),f}$ for any $j,k\in \{1,\dots,m\}$ and any $f\in F^{X(j)}$. We will find these numbers by applying \cref{lem:matrix-rescaling} to the matrix $M\left[\lbrace v(1),\dots,v(m)\rbrace,F^{X(1)}\cup \dots \cup F^{X(m)}\right]$.

In order to apply \cref{lem:matrix-rescaling}, the first thing we need to check is that $M_{v(j),f}>0$ for each $j\in \lbrace 1,\dots,m\rbrace$ and each $f\in (F^{X(1)}\cup \dots \cup F^{X(m)})\setminus F^{X(j)}$. Indeed, since the arcs $X(1),\dots,X(m)\in \mathcal{X}_c$ are all disjoint (as they are well-separated), $v(j)\in X(j)$ cannot be a vertex of the facet $f$, and so we have $M_{v(j),f}>0$.

The other thing we need to check is that for each $j\in \lbrace 1,\dots,m\rbrace$, each $k\in \lbrace 1,\dots,j-1\rbrace$, each $f\in F^{X(j)}$ and each $g\in F^{X(1)}\cup \dots\cup F^{X(j-1)}$, we have $M_{v(j),f}M_{v(k),g}\leq M_{v(j),g}M_{v(k),f}$. But this follows from \cref{lem:circle-geometry} (applied with the arc $X(j)$, the vertices $v(j)$ and $v(k)$ and the facets $f$ and $g$). Indeed, if we write $\eps$ for the length of the arc $X(j)$, then the well-separatedness of the arcs in $\mathcal X_c$, and the fact that we ordered these arcs in order of decreasing size, ensure that $v(k)$ and both endpoints of $g$ have arc-distance at least $5\eps$ from every point of $X(j)$.

We conclude that all assumptions of \cref{lem:matrix-rescaling} are satisfied and we obtain positive real numbers $\alpha_{v(1)},\dots, \alpha_{v(m)}$ satisfying the desired properties.
\end{proof}

\begin{proof}[Proof of \cref{claim:matrix-K-cyclic-polygon}]
We partition the $V\times F_c$ matrix $K$ into submatrices $K[V,F^X]$, for $X\in \mathcal X_c$. It suffices to show that each of these $\vert \mathcal{X}_c\vert$ submatrices has nonnegative entries and nonnegative rank at most $8$. Recall that for each $X\in \mathcal X_c$, the matrix $K[V^X,F^X]$ has only zero entries, so in fact it suffices to consider the submatrices $K[V\setminus V^X,F^X]$. Fix some $X\in \mathcal X_c$; we take a moment to recall all the possibilities for the rows of $K[V\setminus V^X,F^X]$.
\begin{compactitem}
    \item If $v\in V\setminus V_c$ then $K_{v,f}=M_{v,f}$ for all $f\in F^X$.
    \item If $v\in V_c$ and $\phi(v)>\vert V^X\vert$, then $K_{v,f}=\alpha_v M_{v,f}-t^{(\phi(v))}_f=\alpha_v M_{v,f}$ for all $f\in F^X$.
    \item Otherwise, if $v\in V_c$ and $\phi(v)\le \vert V^X\vert$, then $K_{v,f}=\alpha_v M_{v,f}-t^{(\phi(v))}_f=\alpha_v M_{v,f}-\alpha_{x_v} M_{x_v,f}$ for all $f\in F^X$, where $x_v$ is the unique vertex in $V^X$ such that $\phi(x_v)=\phi(v)$.
\end{compactitem}

It follows from this description that the matrix $K[V\setminus V^X,F^X]$ has nonnegative entries. Indeed, each of the entries of $K[V\setminus V^X,F^X]$ is either a positively scaled version of an entry of $M$, or is of the form $\alpha_v M_{v,f}-\alpha_{x_v} M_{x_v,f}$, where $f\in F^X$, and $v\in V_c$ and $x_v\in V^X$ are such that $\phi(v)=\phi(x_v)$. The choice of the numbers $\alpha_v$ in \cref{claim:desired-scaling-matrix} ensures that the latter entries are all nonnegative.

Now, in order to show that $\rank_+ K[V\setminus V^X,F^X]\leq 8$, we apply \cref{lem:shitov-original} to the nonnegative matrix $K[V\setminus V^X,F^X]$, with $V^X=V\cap X$ as our consecutive set of vertices of $P$. The set of facets of $P$ with both endpoints in $V^X$ is precisely the set $F^X$. For each $v\in V\setminus V^X$, in order to choose $\alpha_v$, $\beta_v$ and $x_v$ such that condition ($\star$) is satisfied, we consider the three cases above describing the row of $K[V\setminus V^X,F^X]$ corresponding to $v$. In the first case where $v\in V\setminus V_c$, we can define $\alpha_v=1$, $\beta_v=0$ and take any $x_v\in V^X$. In the second case where $v\in V_c$ and $\phi(v)>\vert V^X\vert$, we have already defined $\alpha_v$ and we can additionally define $\beta_v=0$ and take any $x_v\in V^X$. In the third case where $v\in V_c$ and $\phi(v)\le \vert V^X\vert$, we have already defined $\alpha_v$; as above we let $x_v$ be the unique vertex in $V^X$ such that $\phi(x_v)=\phi(v)$, and let $\beta_v=\alpha_{x_v}$.

So, combining the conclusion of \cref{lem:shitov-original} with the fact that $K[V^X,F^X]$ is the zero matrix, we obtain $\rank_+ K[V,F^X]=\rank_+ K[V\setminus V^X,F^X]\leq 8$, as desired.
\end{proof}

\section{Slacks in cyclic polygons}
\label{section-geometry-circle}

In this section we prove \cref{lem:circle-geometry}. Recall that $X\su \Gamma$ is an arc of length $\eps>0$ and that $Y\su \Gamma$ is the set of all points on the circle $\Gamma$ with arc-distance at least $5\eps$ from every point of $X$. Note that $Y$ is itself an arc of $\Gamma$, and is disjoint from $X$. We are given vertices $v\in X$ and $w\in Y$, a facet $f$ with both endpoints in $X$ and a facet $g$ with both endpoints in $Y$, and our goal is to prove that $M_{v,f}M_{w,g}\leq M_{v,g}M_{w,f}$.

If $M_{w,g}=0$ then the desired inequality is trivially satisfied, so we may assume $M_{w,g}> 0$. Furthermore we have $M_{w,f}> 0$ because $w\in Y$ and both endpoints of $f$ are in $X\su \Gamma\setminus Y$. So our desired inequality is equivalent to $M_{v,f}/M_{w,f}\leq M_{v,g}/M_{w,g}$. Since the entries of the slack matrix $M$ depend on the normalisation of the constraints, it is more convenient to reinterpret this inequality in terms of Euclidean distances. For any point $x\in \Gamma$ and any line $\ell$, let $d(x,\ell)$ denote the Euclidean distance from the point $x$ to the line $\ell$. Let $\ell_f$ and $\ell_g$ be the lines through the facets $f$ and $g$, and note that the desired inequality is equivalent to
\begin{equation}\label{eq-geometry-desired-inequality}
    \frac{d(v,\ell_f)}{d(w,\ell_f)}\leq \frac{d(v,\ell_g)}{d(w,\ell_g)}.
\end{equation}

We now define a point $z_f$ on $X$, in such a way that the ratio of distances $d(v,\ell_f)/d(w,\ell_f)$ can be expressed in terms of Euclidean point-to-point distances $d(v,z_f)$ and $d(w,z_f)$. If the lines $vw$ and $\ell_{f}$ are parallel, then define $z_f$ to be the midpoint of the sub-arc of $X$ between the two endpoints of $f$. Otherwise, if the lines $vw$ and $\ell_f$ intersect in some point $p_f$, then this point $p_f$ must lie on or outside of the circle $\Gamma$ (because $v$, $w$ and both endpoints of the facet $f$ lie on $\Gamma$ and are vertices of the convex polygon $P$). We can therefore consider the lines through $p_f$ tangent to the circle $\Gamma$. Since $\ell_f$ intersects $\Gamma$ in two points on the arc $X$ (namely, the endpoints of $f$), at least one of the tangent lines through $p_f$ touches the circle $\Gamma$ in a point on $X$ (in fact, in a point of the sub-arc of $X$ between the two endpoints of the facet $f$). Define $z_f$ to be such a point.

\begin{claim}\label{claim:distance-ratio}
We have
\[
\frac{d(v,\ell_f)}{d(w,\ell_f)}=\left(\frac{d(v,z_f)}{d(w,z_f)}\right)^2.\]
\end{claim}

\begin{proof}
In the case where ${vw}$ and $\ell_{f}$ are parallel, the point $z_f$ is the midpoint of an arc between the two endpoints of $f$, and also the midpoint of an arc between $v$ and $w$. We therefore have $d(v,\ell_f)=d(w,\ell_f)$ and $d(v,z_f)=d(w,z_f)$, so both sides of the desired equation are equal to 1.

Next, if the lines $vw$ and $\ell_f$ intersect each other in a point $p_f$ which lies on the circle $\Gamma$, then $v$ must be one of the endpoints of the facet $f$ (note that $w\in Y$ cannot be an endpoint of $f$). In this case, we have $p_f=z_f=v$ and $d(v,\ell_f)=d(v,z_f)=0$, so both sides of the desired equation are equal to 0.

Finally, we consider the case where the lines $vw$ and $\ell_{f}$ intersect in a point $p_f$ outside of $\Gamma$ (see \cref{figure-circle-lemma}). In this case, we can first observe that
\[\frac{d(v,\ell_f)}{d(w,\ell_f)}=\frac{d(v,p_f)}{d(w,p_f)},\]
since the two triangles in \cref{figure-circle-lemma} formed by $p_f$ and the two dashed lines are similar. Furthermore, the triangles $p_f v z_f$ and $pz_f w$ are also similar and we obtain
\[\frac{d(v,z_f)}{d(w,z_f)}=\frac{d(v,p_f)}{d(z_f,p_f)}=\frac{d(z_f,p_f)}{d(w,p_f)}.\]
All in all, this yields
\[\left(\frac{d(v,z_f)}{d(w,z_f)}\right)^2=\frac{d(v,p_f)}{d(z_f,p_f)}\cdot \frac{d(z_f,p_f)}{d(w,p_f)}=\frac{d(v,p_f)}{d(w,p_f)}=\frac{d(v,\ell_f)}{d(w,\ell_f)},\]
and finishes the proof of the claim.
\end{proof}

\begin{figure}
\centering
\begin{tikzpicture}[x=2.5cm,y=2.5cm, scale=0.9]
\clip(-2.2,-1.05) rectangle (2.2,1.3);
\draw (0,0) circle (1);
\draw [domain=-2.2:0.9] plot(\x,{(--0.9939598838088064--0.5413719514907043*\x)/1.219601728074889});
\draw [domain=-2.2:1.3] plot(\x,{(--0.5739809988266833--0.3190208262490215*\x)/1.8804156316453997});
\draw [domain=-2.2:1.3] plot(\x,{(--1.5668864051652167--0.8501259075858055*\x)/1.3162138686182372});
\draw (-0.9905445720675394,0.13719129253540555)-- (-0.5425574596750464,0.8400186920247432);
\draw (-0.5425574596750464,0.8400186920247432)-- (0.8898710595778603,0.456212118784427);
\draw [dashed] (-0.9905445720675394,0.13719129253540555)-- (-1.0788382620724908,0.3360991535940264);
\draw [dashed] (0.8898710595778603,0.456212118784427)-- (0.6103498663078566,1.085916944472104);
\draw [line width=2pt] (-0.9120184354907482,0.4101492086119489)-- (0.30758329258414074,0.9515211601026532);
\draw [fill] (-0.9905445720675394,0.13719129253540555) circle (1.5pt);
\draw (-0.9,0.05) node {$v$};
\draw [fill] (0.8898710595778603,0.456212118784427) circle (1.5pt);
\draw (0.82,0.32) node {$w$};
\draw [fill] (-1.8587713282932836,-0.010107215561062354) circle (1.5pt);
\draw (-1.8,-0.15) node {$p_f$};
\draw [fill] (-0.5425574596750464,0.8400186920247432) circle (1.5pt);
\draw (-0.6,0.95) node {$z_f$};
\draw (-0.4,0.52) node {$f$};
\draw (1.0,1.1) node {$\ell_f$};
\end{tikzpicture}
\caption{An illustration for the proof of \cref{claim:distance-ratio} in the case that the lines $vw$ and $\ell_{f}$ intersect in a point $p_f$ outside of $\Gamma$.} 
\label{figure-circle-lemma}
\end{figure}
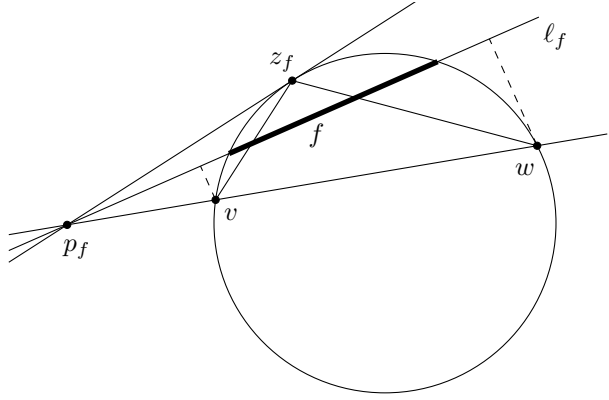

Analogously to our definition of the point $z_f$, we can also define a point $z_g\in Y$ such that
\[\frac{d(v,\ell_g)}{d(w,\ell_g)}=\left(\frac{d(v,z_g)}{d(w,z_g)}\right)^2.\]
The desired inequality \cref{eq-geometry-desired-inequality} is therefore equivalent to
\[\left(\frac{d(v,z_f)}{d(w,z_f)}\right)^2\leq \left(\frac{d(v,z_g)}{d(w,z_g)}\right)^2,\]
which is in turn equivalent to
\begin{equation}\label{eq-geometry-desired-inequality-2}
d(v,z_f)d(w,z_g)\leq d(v,z_g)d(w,z_f).
\end{equation}

To prove this, we first observe that for any points $x\in X$ and $y\in Y$ the Euclidean distance $d(x,y)$ between $x$ and $y$ is at least $3\eps$. Indeed, let $\alpha\leq \pi$ be the arc-distance between $x$ and $y$ (and note that by the definition of the set $Y$ we have $\alpha\geq 5\eps$). Then the angle between the two unit vectors corresponding to $x$ and $y$ (from the origin, which is the centre of $\Gamma$) is precisely $\alpha$ and hence \[d(x,y)=2\cdot \sin(\alpha/2)\geq 2\cdot \frac{\alpha/2}{\pi/2}\geq \frac{10}{\pi}\cdot \eps\geq 3\eps.\]
Here we used that every real number $0\leq t\leq \pi/2$ satisfies $\sin t\geq (\pi/2)^{-1}\cdot t$. Also note that $d(v,z_f)\leq \eps$, since $v$ and $z_f$ lie on the arc $X$, which has length $\varepsilon$. It follows that $d(v,z_g)\geq 3\eps\geq 3\cdot d(v,z_f)$ and $d(w,z_f)\geq 3\eps\geq 3\cdot d(v,z_f)$.

Furthermore, the triangle inequality gives $d(w,z_g)\leq d(w,z_f)+d(v,z_f)+d(v,z_g)$. Hence
\begin{align*}
d(v,z_f)\cdot d(w,z_g)&\leq d(v,z_f)\cdot d(w,z_f)+d(v,z_f)\cdot d(v,z_f)+d(v,z_f)\cdot d(v,z_g)\\
&\leq \frac{1}{3}d(v,z_g)\cdot d(w,z_f)+\frac{1}{3}d(v,z_g)\cdot \frac{1}{3}d(w,z_f)+\frac{1}{3}d(w,z_f)\cdot d(v,z_g)\leq d(v,z_g)\cdot d(w,z_f)
\end{align*}
This proves \cref{eq-geometry-desired-inequality-2} and finishes the proof of \cref{lem:circle-geometry}.

\section{Concluding remarks}

In this paper we proved several results about the extension complexity of low-dimensional polytopes. There are a number of compelling questions left unanswered.

First, we believe it would be interesting to better understand the situation when the dimension is allowed to grow slowly with the number of vertices. We have proved that (for infinitely many $n$) there is an $n^{o(1)}$-dimensional polytope with $n$ vertices and extension complexity $n^{1-o(1)}$, but what can be said about the $o(1)$ terms?
What is the minimum possible dimension $d$ such that there exists a $d$-dimensional polytope with $n$ vertices and extension complexity \emph{exactly} $n$? For example, the cross polytope of dimension $d=n/2$ has $n$ vertices and extension complexity exactly $n$, but is this also possible for a polytope with dimension $d=n^{o(1)}$?
What can be said about the extension complexity of \emph{random} $d$-dimensional $n$-vertex polytopes when $d$ is allowed to grow slowly with $n$?

Second, in \cref{thm:sphere,thm:ball} we have found the typical order of magnitude of $\xc(P)$ for a random $d$-dimensional polytope $P$, for two natural models of random polytopes of a fixed dimension $d$. It may be of interest to analyse and improve the dependence on $d$ in these results. Also, there are various other models of random polytopes one could consider: for example, we could consider $P$ to be the convex hull of random points inside a convex body other than the unit ball, such as a cube or a simplex. We suspect that our methods might still be applicable in such settings, but the requisite geometric considerations may become quite complicated.

Finally, in the case where $d$ is constant, we are still a long way from understanding the maximum possible extension complexity of a $d$-dimensional $n$-vertex polytope. As suggested by \cref{thm:sphere,thm:ball,thm:cyclic-polygon}, could it be that all such polytopes have extension complexity $O(\sqrt n)$? Is there at least an upper bound of the form $o(n)$, for any fixed $d\ge 3$? Shitov has recently made conjectures for both of these questions (see \cite[Conjecture~6.3]{Shi19} and \cite[Conjecture~61]{Sh14v2}). It is tempting to imagine that the arguments in the proofs of \cref{thm:sphere,thm:ball,thm:cyclic-polygon} could be useful in order to prove new upper bounds, though it seems that significant new ideas would be required.

\textbf{Acknowledgements.} We would like to thank Yaroslav Shitov for helpful comments on an earlier version of this paper. We are also grateful to the referees for their careful reading of the paper, and their many useful comments and suggestions. This work started when all three authors were at Stanford University, the first two as Szeg\"o Assistant Professors, and the third as a Visiting Assistant Professor.


\end{document}